\documentclass[a4paper,11pt]{article}
\usepackage[T1]{fontenc}
\usepackage[utf8]{inputenc}
\usepackage{lmodern}
\usepackage{amsmath,amssymb,amsthm,mathtools}
\usepackage[margin=30mm]{geometry}
\usepackage{microtype}
\usepackage{enumitem}
\usepackage[hidelinks]{hyperref}

\numberwithin{equation}{section}

\newtheorem{theorem}{Theorem}[section]
\newtheorem{lemma}[theorem]{Lemma}
\newtheorem{proposition}[theorem]{Proposition}
\newtheorem{corollary}[theorem]{Corollary}
\newtheorem{fact}[theorem]{Fact}

\theoremstyle{definition}
\newtheorem{definition}[theorem]{Definition}
\newtheorem{question}[theorem]{Question}
\theoremstyle{remark}
\newtheorem{remark}[theorem]{Remark}

\newcommand{\ZFC}{\mathsf{ZFC}}
\newcommand{\ZFCm}{\mathsf{ZFC}^{-}}
\newcommand{\CH}{\mathsf{CH}}
\newcommand{\PD}{\mathsf{PD}}
\newcommand{\Code}{\operatorname{Code}}
\newcommand{\side}{\operatorname{side}}
\newcommand{\Init}{\operatorname{Init}}
\newcommand{\Poss}{\operatorname{Poss}}
\newcommand{\lh}{\operatorname{lh}}
\newcommand{\supp}{\operatorname{supp}}
\newcommand{\Sig}{\mathbf{\Sigma}}
\newcommand{\Pil}{\mathbf{\Pi}}
\newcommand{\Del}{\mathbf{\Delta}}
\newcommand{\fs}{\mathrm{fs}}
\newcommand{\cs}{\mathrm{cs}}
\newcommand{\R}{\mathbb{R}}
\newcommand{\AD}{\mathbb{A}_D}
\newcommand{\Cohen}{\mathbb{C}}
\newcommand{\PP}{\mathbb{P}}
\newcommand{\QQ}{\mathbb{Q}}
\newcommand{\TT}{\mathbb{T}}
\newcommand{\BB}{\mathbb{B}}
\newcommand{\res}{\mathbin{\restriction}}
\newcommand{\la}{\langle}
\newcommand{\ra}{\rangle}
\newcommand{\forces}{\Vdash}
\newcommand{\WO}{\mathsf{WO}}
\newcommand{\SEP}{\mathsf{SEP}}
\newcommand{\AUX}{\mathsf{AUX}}

\title{Forcing ${\Sigma}^1_3$-Separation and a
  $\Delta^1_3$ Well-Order}
\author{Stefan Hoffelner\thanks{TU Wien, Wiedner Hauptstra\ss{}e
8--10/104, 1040 Wien, Austria. \texttt{stefan.hoffelner@tuwien.ac.at}.
This research was funded in whole by the Austrian Science Fund (FWF),
Grant-DOI 10.55776/P37228.}}
\date{\today}

\begin{document}
\maketitle

\begin{abstract}
We construct a cardinal-preserving forcing extension of $L$ in which
$\CH$ holds, the reals have a parameter-free $\Delta^1_3$ well-order,
and the $\Sigma^1_3$-separation property holds both in its boldface and
in its lightface form. In this extension there is no good $\Delta^1_3$
well-order of the reals, and $\Pil^1_3$-reduction and
$\Pil^1_3$-uniformization fail.
\end{abstract}


\section{Introduction}

The separation property and its relatives, the reduction and the
uniformization property, are among the oldest structural properties of
definable sets of reals. For a projective pointclass $\Gamma$ we write
$\check\Gamma$ for the dual pointclass and $\Delta=\Gamma\cap\check\Gamma$.

\begin{definition}\label{def:sep}
Let $\Gamma$ be a projective pointclass. The following properties are
denoted by
$\operatorname{Sep}(\Gamma)$, $\operatorname{Red}(\Gamma)$ and
$\operatorname{Unif}(\Gamma)$, respectively.
\begin{enumerate}[label=(\arabic*)]
\item \emph{Separation:} for $A_0,A_1\in\Gamma$ with
$A_0\cap A_1=\varnothing$, there is $D\in\Delta$ such that
$A_0\subseteq D\subseteq\R\setminus A_1$.
\item \emph{Reduction:} for $A_0,A_1\in\Gamma$ there are
$B_0,B_1\in\Gamma$ with
$
B_i\subseteq A_i\ (i<2),\qquad B_0\cap B_1=\varnothing,\qquad
B_0\cup B_1=A_0\cup A_1.
$
\item \emph{Uniformization:} for every $R\subseteq\R^2$ in $\Gamma$ there is
a function $f$ with $\operatorname{dom}(f)=\operatorname{pr}_1(R)$ and
$\operatorname{graph}(f)\subseteq R$, $\operatorname{graph}(f)\in\Gamma$.
\end{enumerate}
\end{definition}

These notions apply to the boldface pointclasses $\Sig^1_n$ and
$\Pil^1_n$, whose sets are defined with real parameters, and to the
lightface pointclasses $\Sigma^1_n$ and $\Pi^1_n$, whose sets are defined
without parameters. Unless we say otherwise, we refer to the boldface
versions.

Uniformization for $\Gamma$ implies reduction for $\Gamma$, which in turn
implies separation for $\check\Gamma$; moreover, reduction for $\Gamma$
excludes separation for $\Gamma$ (Fact~\ref{fact:classical}). By Kond\^o's
theorem (see \cite{Moschovakis,Kechris}) $\Pil^1_1$ and $\Sig^1_2$ have the
uniformization property, so
$\Sig^1_1$ and $\Pil^1_2$ have the separation property, and this is all
that $\ZFC$ decides. Beyond the second level two very different pictures
emerge. In $L$ the canonical good $\Delta^1_2$ well-order of the reals
yields $\Sig^1_n$-uniformization for all $n\geq 2$ by Addison's theorem
\cite{Addison}, hence $\Pil^1_n$-separation and the failure of
$\Sig^1_n$-separation for every $n\geq 2$. Under projective determinacy,
on the other hand, $\Pil^1_{2n+1}$ and $\Sig^1_{2n+2}$ have the scale
property \cite{Moschovakis}, so the pattern alternates; in particular
$\Sig^1_3$-separation holds, and already $\Del^1_2$-determinacy implies
$\Pil^1_3$-uniformization \cite{Moschovakis71}. Projective determinacy also
rules out projective well-orders of the reals.

Mathias \cite[Problem 3029]{Mathias} asked whether $\Sig^1_3$-separation
can be forced over a model of $\ZFC$. This was answered in \cite{HoSep},
where a universe satisfying $\Sig^1_3$-separation is obtained by forcing
over $L$.
A universe satisfying lightface $\Sigma^1_3$-separation is constructed
there as well.
Subsequently, the method was used to separate $\Pil^1_3$-reduction from
$\Pil^1_3$-uniformization \cite{HoRedUnif}, to force
$\Pil^1_n$-uniformization \cite{HoPiUnif}, and to obtain
$\Sig^1_3$-separation together with a failure of $\Pil^1_3$-reduction
\cite{HoFailRed}. Projective well-orders of the reals in forcing
extensions of $L$, on the other hand, go back to Harrington's long
$\Delta^1_3$ well-orderings \cite{Harrington}; see
\cite{FischerFriedman,FFZ,CaicedoFriedman} for later developments. In
\cite{HoLC}, global $\Sigma$-uniformization is combined with a large
continuum and a projective well-order of the reals.
In \cite{HoPiSide}, $\Sig^1_{n+3}$-separation and a lightface
$\Delta^1_{n+3}$ well-order of the reals are obtained by forcing over the
canonical inner model $M_n$ with $n\geq1$ Woodin cardinals, together with
Lebesgue measurability and the Baire property of all $\Sig^1_{n+2}$ sets.

In this paper we combine $\Sigma^1_3$-separation, the weakest of the three
structural properties which $\PD$ yields at the third projective level,
with a $\Delta^1_3$-definable well-order of the reals. We obtain the
separation property in both of its forms. \emph{Boldface}
$\Sig^1_3$-separation says that any two disjoint $\Sig^1_3$ sets can be
separated by a $\Del^1_3$ set. \emph{Lightface} $\Sigma^1_3$-separation
says that any two disjoint $\Sigma^1_3$ sets can be separated by a
$\Delta^1_3$ set, where neither the two given sets nor the separating set
are defined with parameters. There does not seem to be an easy
implication between the two properties in either direction: for a pair of
sets defined without parameters, the boldface property only provides a
separating set whose definition may involve a real parameter, and the
lightface property says nothing about sets which are defined with
parameters. We therefore state both.

\begin{theorem}\label{thm:main-intro}
There is a cardinal-preserving forcing extension of $L$ in which the
following hold.
\begin{enumerate}[label=(\arabic*)]
\item $\CH$.
\item Boldface $\Sig^1_3$-separation: any two disjoint $\Sig^1_3$ sets of
reals can be separated by a $\Del^1_3$ set.
\item Lightface $\Sigma^1_3$-separation: any two disjoint $\Sigma^1_3$
sets of reals can be separated by a $\Delta^1_3$ set.
\item There is a parameter-free $\Delta^1_3$-definable well-order of the
reals.
\end{enumerate}
\end{theorem}

The distinction between a projective well-order and a \emph{good}
projective well-order is essential here (Definition~\ref{def:good}). By
Addison's theorem, a good $\Delta^1_3$ well-order would imply
$\Sig^1_3$-uniformization, hence $\Sig^1_3$-reduction, which excludes
$\Sig^1_3$-separation. Consequently the model of
Theorem~\ref{thm:main-intro} has no good $\Delta^1_3$ well-order. The
well-order we construct has order type greater than $\omega_1$: the reals
of the ground model form an initial segment of size $\aleph_1$, and the
new reals are placed above it. See Remark~\ref{rem:good}.

Moreover, at every level $k\geq2$ a $\Del^1_k$ well-order of the reals
rules out $\Pil^1_k$-reduction \cite[Theorem~2.1]{HoPiSide}. We obtain
the following consequences (Corollary~\ref{cor:sep-wo}).

\begin{corollary}\label{cor:main-intro}
In the extension of Theorem~\ref{thm:main-intro}, there is no good
$\Delta^1_3(a)$ well-order of the reals for any real $a$, and
$\Pil^1_3$-reduction and $\Pil^1_3$-uniformization fail.
\end{corollary}

In particular, Theorem~\ref{thm:main-intro} provides a further model of
$\Sig^1_3$-separation in which $\Pil^1_3$-reduction fails, compare
\cite{HoFailRed}.

\subsection*{Organization}

Section~\ref{sec:sep-red-wo} collects the facts on separation, reduction
and well-orders which we use. In Section~\ref{sec:coding} we recall the
ground model and the coding of \cite{HoFailRed} and show that exactly the
coded reals satisfy the decoding formula. Section~\ref{sec:allowable}
introduces stages, end-extension requirements and allowable forcings, and
proves closure under finite products and invariance of the tests for
allowable forcings. In Section~\ref{sec:iteration} we fix the requirements
for the parameter-free pairs and define the iteration, and in
Section~\ref{sec:verification} we prove Theorem~\ref{thm:main-intro} and
Corollary~\ref{cor:main-intro}. We close with questions in
Section~\ref{sec:questions}.

\subsection*{Notation}

We use $\R:=2^\omega$ as the space of reals; the corresponding statements
for $\omega^\omega$ follow via a $\Delta^1_1$ bijection between $2^\omega$
and $\omega^\omega$. For a model $M$ we
write $\R^M$ for the set of reals of $M$. Boldface pointclasses allow real
parameters, lightface ones do not. Finite tuples of reals and natural
numbers are coded into reals by a fixed recursive coding
$\la\cdot,\ldots,\cdot\ra$, and for a real $x$ we let $(x)_i$ denote the
$i$-th real recursively decoded from $x$. $\ZFCm$ denotes $\ZFC$ without
the power set axiom, with Collection instead of Replacement. We write
$\mathbf 0$ for the element of $2^\omega$ with constant value $0$.

A finite support iteration of length $\delta$ is written
$\PP=(\PP_\gamma,\dot\PP(\gamma):\gamma<\delta)$; $\PP_\delta$ is the
whole iteration and we identify $\PP$ with $\PP_\delta$ when no confusion
arises. If $G$ is $\PP_\delta$-generic, $G_\gamma:=G\cap\PP_\gamma$ and
$G(\gamma)$ is the generic for the $\gamma$-th factor. For
$\gamma\leq\gamma'\leq\delta$ the tail $\PP_{[\gamma,\gamma')}$ is the
$\PP_\gamma$-name for the quotient, so that
$\PP_{\gamma'}\cong\PP_\gamma*\PP_{[\gamma,\gamma')}$. For a forcing $\PP$
and $p\in\PP$ we let $\PP\res p=\{q\in\PP:q\leq p\}$. Dots on names are
omitted when no confusion arises. For well-orders $\prec$ of $X$ and
$\prec'$ of $X'$, put
\begin{align*}
\Init(\prec)&:=\{F\subseteq X:\forall x\in F\ \forall y\in X\,
(y\prec x\Rightarrow y\in F)\},\\
\prec\unlhd\prec'&\quad:\Longleftrightarrow\quad
X\subseteq X',\quad \prec'\res X=\prec,\quad X\in\Init(\prec').
\end{align*}
Thus $\prec\unlhd\prec'$ means that $\prec'$ end-extends $\prec$.

\section{Separation, reduction and well-orders}\label{sec:sep-red-wo}

\begin{fact}\label{fact:classical}
Let $\Gamma$ be $\Sig^1_n$ or $\Pil^1_n$ for some $n\geq 1$.
\begin{enumerate}[label=(\alph*)]
\item $\operatorname{Unif}(\Gamma)\Rightarrow\operatorname{Red}(\Gamma)$.
\item $\operatorname{Red}(\Gamma)\Rightarrow\operatorname{Sep}(\check\Gamma)$.
\item $\operatorname{Red}(\Gamma)\Rightarrow\neg\operatorname{Sep}(\Gamma)$.
\end{enumerate}
\end{fact}

\begin{proof}
These are the classical implications between uniformization, reduction
and separation, including Novikov's theorem that reduction and separation
cannot both hold for the same pointclass; see
\cite[Chapter~4]{Moschovakis}.
\end{proof}

\begin{definition}\label{def:good}
Let $a$ be a real. A $\Delta^1_n(a)$ well-order $\prec$ of $\R$ is
\emph{good} if $\operatorname{otp}(\R,\prec)=\omega_1$ and
\[
\mathrm{IS}(x,y)\quad:\Longleftrightarrow\quad
\{(x)_i:i<\omega\}=\{z:z\prec y\}
\]
is $\Sigma^1_n(a)$.
\end{definition}

\begin{fact}[Addison \cite{Addison}]\label{fact:addison}
Let $n\geq2$. If there is a good $\Delta^1_n(a)$ well-order of $\R$ for
some real $a$, then $\operatorname{Unif}(\Sig^1_m)$ holds for every
$m\geq n$.
\end{fact}

The following fact is proven in $\ZFC$ in \cite{HoPiSide}; the
well-order may have any order type.

\begin{fact}[{\cite[Theorem~2.1]{HoPiSide}}]\label{fact:red-wo}
Let $k\geq2$. If there is a $\Del^1_k$ well-order of $\R$, then
$\Pil^1_k$-reduction fails.
\end{fact}

\begin{corollary}\label{cor:sep-wo}
Let $k\geq2$, and suppose that $\Sig^1_k$-separation holds and that there
is a $\Del^1_k$ well-order of $\R$. Then there is no good $\Delta^1_k(a)$
well-order of $\R$ for any real $a$, and $\Pil^1_k$-reduction and
$\Pil^1_k$-uniformization fail.
\end{corollary}

\begin{proof}
A good $\Delta^1_k(a)$ well-order would give
$\operatorname{Unif}(\Sig^1_k)$ by Fact~\ref{fact:addison}, hence
$\operatorname{Red}(\Sig^1_k)$ and therefore
$\neg\operatorname{Sep}(\Sig^1_k)$ by Fact~\ref{fact:classical}(a),(c). By Fact~\ref{fact:red-wo},
$\Pil^1_k$-reduction fails, and so does $\Pil^1_k$-uniformization by
Fact~\ref{fact:classical}(a).
\end{proof}

\section{The ground model and the coding}\label{sec:coding}

We use the ground model and the coding method of \cite{HoFailRed}. In
this section we recall them, and we show that, in all models we consider,
exactly the coded reals satisfy the decoding formula
(Lemma~\ref{lem:coding}).

\subsection{Suslin trees and almost disjoint coding}

\paragraph{Suslin trees.}
All trees below are normal, splitting $\omega_1$-trees whose underlying
set is a subset of $\omega_1$.

\begin{definition}
A sequence $\vec S=(S_\xi:\xi<\kappa)$ of Suslin trees is
\emph{independent} if $\prod_{\xi\in e}S_\xi$ is Suslin for every
nonempty finite $e\subseteq\kappa$.
\end{definition}

\begin{fact}[{\cite[Theorem~2.2]{HoFailRed}}]\label{fact:trees}
There is an independent sequence $\vec S=(S_\xi:\xi<\omega_1)$ in $L$,
uniformly $\Sigma_1(\{\omega_1\})$-definable as follows. For a transitive
$N\models\ZFCm$ with $\omega_1^N=(\omega_1^L)^N$, let $\vec S^{\,N}$
be the sequence computed by the fixed defining formula. Then
\[
\vec S^{\,N}=
\begin{cases}
\vec S,&N\text{ uncountable},\\
(S_\xi\res\beta:\xi<\beta),
 &N\text{ countable},\ \beta=\omega_1^N,
\end{cases}
\]
where $S_\xi\res\beta$ consists of the levels below $\beta$.
\end{fact}

The sequence is obtained from the canonical $\diamondsuit$-sequence by the
usual construction of a Suslin tree, sealing off antichains in all finite
products simultaneously. We fix $\vec S$ from now on.

\begin{fact}\label{fact:suslin-preservation}
Let $T$ be Suslin and $\vec S=(S_\xi:\xi<\kappa)$ independent.
\begin{enumerate}[label=(\alph*)]
\item If $\PP$ is $\sigma$-closed, then $\PP$ forces that $T$ is Suslin;
in particular, $\PP$ preserves the independence of $\vec S$.
\item For $I\subseteq\kappa$, put $\PP_I=\prod^{\fs}_{\xi\in I}S_\xi$.
Then $\PP_I$ is ccc and forces that $\prod_{\xi\in e}S_\xi$ is Suslin for
every nonempty finite $e\subseteq\kappa\setminus I$.
\item If $\PP$ is Knaster, then $\PP$ forces that $T$ is Suslin.
Finite support iterations of Knaster forcings are Knaster; this applies
in particular to finite support iterations of $\sigma$-centered forcings.
\end{enumerate}
\end{fact}

\begin{proof}
For (a), use the standard preservation theorem for $\sigma$-closed
forcing. For (b), let $e\subseteq\kappa\setminus I$ be nonempty and
finite. A $\Delta$-system argument shows that
$\PP_I\times\prod_{\xi\in e}S_\xi$ is ccc, so $\PP_I$ forces that
$\prod_{\xi\in e}S_\xi$ is ccc, i.e., Suslin. For (c), if $\PP$ is
Knaster, then $\PP\times T$ is ccc, so $\PP$ forces that $T$ is ccc; and
finite support iterations preserve the Knaster property. See
\cite{Jech,Kunen} and \cite[Section~2]{HoFailRed}.
\end{proof}

\paragraph{Almost disjoint coding.}
Fix in $L$ an almost disjoint family $D=(d_\zeta:\zeta<\omega_1)$ of
infinite subsets of $\omega$ which is definable in the same uniform way as
$\vec S$, see \cite{JensenSolovay,HoFailRed}.

\begin{definition}
For $Y\subseteq\omega_1$, the almost disjoint coding forcing $\AD(Y)$ has
underlying set $[\omega]^{<\omega}\times[\omega_1]^{<\omega}$ and is
ordered by
\[
(t,H)\leq_Y(s,F)\quad\Longleftrightarrow\quad
s\subseteq t,\quad F\subseteq H,\quad
\forall\zeta\in F\cap Y\ ((t\setminus s)\cap d_\zeta=\varnothing).
\]
\end{definition}

The underlying set of $\AD(Y)$ does not depend on $Y$, while the order
does.

\begin{fact}[\cite{JensenSolovay}]\label{fact:ad}
Let $Y\subseteq\omega_1$, $e\in\AD(Y)$, and $G$ be generic for
$\AD(Y)\res e$ over $M$. Put $r=\bigcup\{s:(s,F)\in G\}$.
\begin{enumerate}[label=(\alph*)]
\item $\AD(Y)\res e$ is $\sigma$-centered: conditions with the same
first coordinate are compatible.
\item $Y=\{\zeta<\omega_1:|r\cap d_\zeta|<\aleph_0\}$.
\item $\AD(\varnothing)$ is directed and forcing equivalent to the
trivial forcing.
\item If $Y$ is infinite, then $r\notin M$.
\end{enumerate}
\end{fact}

\begin{proof}
(a)--(c) are immediate. For (d), fix $z\in\R^M$ and $(s,F)\leq e$.
As $Y$ is infinite and $D$ is almost disjoint, there are
\[
\zeta\in Y\setminus F,\qquad
n\in d_\zeta\setminus
\bigl(s\cup\textstyle\bigcup_{\zeta'\in F\cap Y}d_{\zeta'}\bigr).
\]
Set
\[
q:=\begin{cases}
(s\cup\{n\},F),&n\notin z,\\
(s,F\cup\{\zeta\}),&n\in z.
\end{cases}
\qquad
q\leq(s,F),\quad q\forces n\in r\mathbin\triangle z.
\]
Thus $\{q:q\forces r\neq z\}$ is dense for every $z\in\R^M$.
\end{proof}

\paragraph{Absoluteness and homogeneity.}
We use Shoenfield absoluteness in the following form: if $M\subseteq N$
are transitive models of $\ZFC$ and $N$ is a forcing extension of $M$, then
$\Sigma^1_2$ statements with parameters in $M$ have the same truth value in
$M$ and in $N$; consequently, $\Sigma^1_3$ statements with parameters in
$M$ which hold in $M$ also hold in $N$ \cite[Theorem~25.20]{Jech}. We use
Easton's lemma \cite[Lemma~15.19]{Jech}: if $\PP$ is ccc and $\QQ$ is
$\sigma$-closed, then $\QQ$ is $\sigma$-distributive in every
$\PP$-generic extension. Finally, we use the automorphism lemma: if $\pi$
is an automorphism of a forcing
$\PP$, then $p\forces\varphi(\tau)$ iff $\pi p\forces\varphi(\pi\tau)$ for
all formulas $\varphi$ and $\PP$-names $\tau$ (see \cite[Chapter~14]{Jech}).

\subsection{The ground model and the coding areas}

Let $\Cohen(\omega_1)=\{p\in L:p\colon\beta\to2,\ \beta<\omega_1\}$,
ordered by end-extension, and put
\[
\QQ^0:=\Big(\prod^{\fs}_{\xi<\omega_1}S_\xi\Big)^L,
\qquad
\QQ^1:=\Big(\prod^{\cs}_{\eta<\omega_1}\Cohen(\omega_1)\Big)^L.
\]
Thus the conditions of $\QQ^1$ are the countable partial functions in $L$
from $\omega_1$ to $\Cohen(\omega_1)$.
Fix $\QQ^0\times\QQ^1$-generic $G^0\times G^1$ over $L$, and let
\[
W:=L[G^0][G^1],\qquad
b_\xi:=G^0(\xi),\qquad g_\eta:=\bigcup G^1(\eta).
\]
Here $b_\xi$ is the cofinal branch through $S_\xi$ and
$g_\eta\colon\omega_1\to2$ the Cohen subset at coordinate $\eta$.
This is the ground model of \cite{HoFailRed}.

For $J\in L$, $J\subseteq\omega_1$, write
$\Cohen(J):=\{p\in\QQ^1:\supp(p)\subseteq J\}$ and $G^1\res J$ for the
restricted generic; use $G^0\res I$ analogously. For countable
$V\in L$, $V\subseteq\omega_1$, and $\xi<\omega_1$, put
\[
W\res V:=L[G^0][G^1\res V],\qquad
W^{(\xi)}:=L[G^0\res(\omega_1\setminus\{\xi\})][G^1].
\]

\begin{lemma}\label{lem:W}
\begin{enumerate}[label=(\alph*)]
\item $W$ has the same cardinals as $L$,
$W\models\CH$, and $\R^W=\R^{L[G^0]}$.
\item If $V\in L\cap[\omega_1]^{\leq\omega}$ and
$H_V:=G^1\res(\omega_1\setminus V)$, then
$H_V$ is $\Cohen(\omega_1\setminus V)$-generic over $W\res V$ and
$W=(W\res V)[H_V]$.
Moreover, $\Cohen(\omega_1\setminus V)$ is $\sigma$-distributive in every
ccc extension of $W\res V$.
\item For $\xi<\omega_1$, $W^{(\xi)}\models``S_\xi\text{ is Suslin}''$,
$b_\xi$ is $S_\xi$-generic over $W^{(\xi)}$, and
\[
W=W^{(\xi)}[b_\xi],\qquad
[\operatorname{Ord}]^{\leq\omega}\cap W^{(\xi)}
=[\operatorname{Ord}]^{\leq\omega}\cap W.
\]
\end{enumerate}
\end{lemma}

\begin{proof}
(a) In $L$, $\QQ^1$ is $\sigma$-closed and $|\QQ^1|=\aleph_1$; hence
$L[G^1]$ has the same cardinals and the same countable sequences of
ordinals as $L$, and $L[G^1]\models\CH$. By Fact~\ref{fact:suslin-preservation}(a),(b),
$L[G^1]\models``\vec S\text{ is independent}''$ and
$L[G^1]\models``\QQ^0\text{ is ccc}''$,
with $|\QQ^0|=\aleph_1$.
Thus $W=L[G^1][G^0]$ has the same cardinals as $L[G^1]$ and satisfies
$\CH$. Easton's lemma gives
$L[G^0]\models``\QQ^1\text{ is $\sigma$-distributive}''$, so
$\R^W=\R^{L[G^0]}$.

(b) Use $\QQ^1\cong\Cohen(V)\times\Cohen(\omega_1\setminus V)$ in $L$
and the product lemma. Since $L[G^1\res V]$ has no new countable sequences
of elements of $L$,
$\Cohen(\omega_1\setminus V)$ remains $\sigma$-closed in $L[G^1\res V]$.
Every ccc extension of $W\res V=L[G^1\res V][G^0]$ is a ccc extension of
$L[G^1\res V]$; Easton's lemma yields the last assertion.

(c) In $L[G^1]$,
$\QQ^0\cong S_\xi\times\prod^{\fs}_{\zeta\ne\xi}S_\zeta$.
Fact~\ref{fact:suslin-preservation}(b) and the product lemma give
$W^{(\xi)}\models``S_\xi\text{ is Suslin}''$ and $W=W^{(\xi)}[b_\xi]$.
Normal Suslin trees are $\omega$-distributive, so $W^{(\xi)}$ and $W$ have
the same countable sequences of ordinals.
\end{proof}

\paragraph{Coding areas.}
Fix a bijection $\vartheta\colon(2^{<\omega_1})^L\to\omega_1$ in $L$ and put
\[
h_\eta:=\{\vartheta(g_\eta\res\beta):\beta<\omega_1\}
\qquad(\eta<\omega_1).
\]
We call $g_\eta$, or $h_\eta$, the \emph{coding area} with index $\eta$.

\begin{lemma}\label{lem:areas}
For $\eta\ne\eta'<\omega_1$,
\[
\sup h_\eta=\omega_1,\qquad
|h_\eta\cap h_{\eta'}|\leq\aleph_0,
\quad\text{hence}\quad\sup(h_\eta\cap h_{\eta'})<\omega_1.
\]
\end{lemma}

\begin{proof}
For each $\beta<\omega_1$, some $p\in G^1$ has
$g_\eta\res\beta\subseteq p(\eta)\in L$; hence
$g_\eta\res\beta\in L$ and injectivity of $\vartheta$ gives
$|h_\eta|=\aleph_1$. If $\eta\ne\eta'$, genericity gives
$g_\eta(\beta_0)\ne g_{\eta'}(\beta_0)$ for some $\beta_0<\omega_1$.
Then
\[
g_\eta\res\beta=g_{\eta'}\res\beta'
\ \Longrightarrow\ \beta=\beta'\leq\beta_0.
\]
Consequently
\[
h_\eta\cap h_{\eta'}\subseteq
\{\vartheta(g_\eta\res\beta):\beta\leq\beta_0\}.
\]
\end{proof}

For $w\in2^\omega$ and $\eta<\omega_1$, let
\[
I(\eta,w):=\{\omega\gamma+2n+w(n):\gamma\in h_\eta,\ n<\omega\}.
\]
Thus, for each $\gamma\in h_\eta$ and $n<\omega$, precisely
$S_{\omega\gamma+2n+w(n)}$ is selected from the pair
$S_{\omega\gamma+2n},S_{\omega\gamma+2n+1}$.

\subsection{The coding forcings and the decoding formula}

Let $X(\eta,w)\subseteq\omega_1$ code
$\la w,g_\eta,(b_\xi:\xi\in I(\eta,w))\ra$ in a fixed absolute way, that
is, by the same formula in every transitive $\ZFC$-model containing these
objects.

Write $\mathrm{Br}(T)$ for the assertion that $T$ has a cofinal branch.
For a transitive set $N$, an ordinal $\Omega$ and a real $w$, let
$\Theta(N,\Omega,w)$ assert:
\begin{enumerate}[label=(\roman*)]
\item $N\models\ZFCm$, $\Omega,w\in N$, and
$\omega_1^N=(\omega_1^L)^N=\Omega$;
\item there is $h\in N$ with $h\subseteq\Omega$ and $\sup h=\Omega$ such
that $N\models\forall\gamma\in h\ \forall n<\omega\,
\mathrm{Br}(S_{\omega\gamma+2n+w(n)})$.
\end{enumerate}
The trees are computed in $N$ by the formula of Fact~\ref{fact:trees}. Let
$\mathcal M(r,w)$ be the class of countable transitive $M$ such that
$M\models\ZFCm+``\omega_1\text{ exists}''$,
$\omega_1^M=(\omega_1^L)^M$ and $r,w\in M$,
and put
\[
\Psi(r,w)\iff\forall M\in\mathcal M(r,w)\
 [M\models\exists N\in L[r]\ \Theta(N,\omega_1,w)],
\qquad
\Phi(w)\iff\exists r\ \Psi(r,w).
\]
Coding countable transitive models by reals gives
$\Psi\in\Pi^1_2$ and $\Phi\in\Sigma^1_3$, both lightface.

\begin{fact}\label{fact:localization}
There is a formula which defines, uniformly in every transitive
$\ZFC$-model which contains $X(\eta,w)$ and satisfies
$\omega_1=\omega_1^L$, a set $Y(\eta,w)\subseteq\omega_1$ such that
$Y(\eta,w)\in L[X(\eta,w)]$ and, for every real $r$,
\[
Y(\eta,w)=\{\zeta<\omega_1:|r\cap d_\zeta|<\aleph_0\}
\ \Longrightarrow\ \Psi(r,w).
\]
\end{fact}

This is David's trick \cite{David}, in the form used in
\cite{HoSep,HoFailRed}. In $L[X]$, $X=X(\eta,w)$, choose a transitive
$\ZFCm$-model of size $\aleph_1$ containing $X$ and $\omega_1$, and a
continuous chain $(N_i:i<\omega_1)$ of countable elementary submodels of
it, and let $C:=\{N_i\cap\omega_1:i<\omega_1\}$, a club in $\omega_1$.
Let $Y$ code $X$ and, for each $\beta\in C$, a real which codes a
well-order of $\omega$ whose order type is the next element of $C$. Let
$M\in\mathcal M(r,w)$ and $\beta:=\omega_1^M$. Then $M$ decodes
$Y\cap\beta$ from $r$, the reals coded into $Y\cap\beta$ show that
$\beta\in C$, and
\[
Y\cap\beta\in L[r]^M,\qquad
M\models\exists N\in L[Y\cap\beta]\ \Theta(N,\beta,w).
\]
The witness $N$ is the transitive collapse of the submodel $N_i$ with
$N_i\cap\omega_1=\beta$; it inherits clause~(ii) by elementarity from the
model of size $\aleph_1$, in which $h_\eta$ and the branches coded in $X$
witness it. Note also that $\sup Y(\eta,w)=\omega_1$.

\begin{definition}
$\Code(w,\eta):=\AD(Y(\eta,w))$ for $w\in2^\omega$, $\eta<\omega_1$.
\end{definition}

If $N\supseteq L$ is a $\ZFC$-model with $\omega_1^N=\omega_1^L$ and
$w,g_\eta,(b_\xi:\xi\in I(\eta,w))\in N$, then $\Code(w,\eta)$ can be
computed in $N$. In particular, it can be computed in $W\res V$ when
$\eta\in V$, and in $W^{(\xi)}$ when $\xi\notin I(\eta,w)$.

Fix three distinct natural numbers $\WO$, $\SEP$ and $\AUX$. We code
three kinds of tuples; under the fixed recursive coding, tuples
with distinct first entries are coded by distinct reals:
\begin{itemize}
\item $\la\WO,u,v\ra$: $u$ precedes $v$ in the well-order;
\item $\la\SEP,c,y,m,k,x,i\ra$, $c,y,x\in\R$, $m,k<\omega$, $i<2$:
$x$ is on side $i$ for the requirement $(c,y,m,k)$;
\item $\la\AUX\ra$: an auxiliary code.
\end{itemize}

\subsection{Coding iterations}

\begin{definition}\label{def:coding-iteration}
Let $N\supseteq L$ be a $\ZFC$-model with $\omega_1^N=\omega_1^L$.
A \emph{coding iteration over $N$} is a finite support iteration
$\PP=(\PP_\gamma,\dot\PP(\gamma):\gamma<\delta)\in N$, $\delta\leq\omega_1$,
with $\PP_\gamma$-names
\[
\dot w_\gamma\in2^\omega,\quad\dot\eta_\gamma<\omega_1,
\quad\dot a_\gamma\in2,\quad
\dot q_\gamma\in[\omega]^{<\omega}\times[\omega_1]^{<\omega},
\]
such that $\PP_\gamma$ forces:
\begin{enumerate}[label=(\roman*)]
\item $a_\gamma=1\ \Longrightarrow\
\dot\PP(\gamma)=\Code(w_\gamma,\eta_\gamma)\res q_\gamma$, with
$\Code(w_\gamma,\eta_\gamma)$ computable in $N[G_\gamma]$ and
$\eta_\gamma\notin\{\eta_{\gamma'}:\gamma'<\gamma,\ a_{\gamma'}=1\}$;
\item $a_\gamma=0\ \Longrightarrow\
\dot\PP(\gamma)=\AD(\varnothing)\res q_\gamma$.
\end{enumerate}
When $a_\gamma=1$, step $\gamma$ is \emph{active}, $w_\gamma$ is its
\emph{coded real}, $\eta_\gamma$ its \emph{area}, and $r_\gamma$ its
generic real.
\end{definition}

For $p\in\PP_\delta$ we identify $\PP_\delta\res p$ with the coding
iteration whose factors are $\dot\PP(\gamma)\res p(\gamma)$,
$\gamma<\delta$. A condition $p$ is \emph{determined} if for every
$\gamma\in\supp(p)$, $p(\gamma)$ is the check name of a pair in
$[\omega]^{<\omega}\times[\omega_1]^{<\omega}$. Let $D_\PP$ be the set of
determined conditions:
\[
D_\PP:=\{p\in\PP_\delta:\forall\gamma\in\supp(p)\
 \exists e\in[\omega]^{<\omega}\times[\omega_1]^{<\omega}\
 [p(\gamma)=\check e]\}.
\]

\begin{lemma}\label{lem:coding-iteration-basic}
Let $G$ be generic for a coding iteration $\PP$ over $W$ of length
$\delta\leq\omega_1$.
\begin{enumerate}[label=(\alph*)]
\item $D_\PP$ is dense and $|D_\PP|\leq\aleph_1$.
\item $\PP_\delta$ is ccc, preserves cardinals, and
$W[G]\models\CH$.
\item If $\delta=\omega_1$, then
$\R^{W[G]}=\bigcup_{\gamma<\omega_1}\R^{W[G_\gamma]}$.
\item For every active $\gamma<\delta$,
$r_\gamma\notin W[G_\gamma]$ and
$Y(\eta_\gamma,w_\gamma)=\{\zeta:|r_\gamma\cap d_\zeta|<\aleph_0\}$.
\item Every name for a real is equivalent to a nice name
$\tau\subseteq\omega\times D_\PP$ with $|\tau|\leq\aleph_0$.
\end{enumerate}
\end{lemma}

\begin{proof}
(a) By induction on $\max\supp(p)$: given $p\in\PP_{\gamma+1}$,
strengthen $p\res\gamma$ to decide $p(\gamma)=\check e$ for some pair $e$,
and then make the part below $\gamma$ determined by the induction
hypothesis. Thus $D_\PP$ is dense, and
\[
D_\PP\subseteq
\{t:t\colon\delta\rightharpoonup
[\omega]^{<\omega}\times[\omega_1]^{<\omega},\ |t|<\aleph_0\},
\qquad |D_\PP|\leq\aleph_1.
\]
(b) Each factor is $\sigma$-centered, so $\PP_\delta$ is ccc
\cite[Theorem~16.30]{Jech}. By (a) and $W\models\CH$,
$|[\omega\times D_\PP]^{\leq\omega}|\leq\aleph_1$, giving $\CH$ in $W[G]$.
(c),(e) By ccc and (a), nice names are countable; when $\delta=\omega_1$,
$\sup\bigcup\{\supp(p):(n,p)\in\tau\}<\omega_1$.
(d) Apply Fact~\ref{fact:ad}(b),(d), since
$\sup Y(\eta_\gamma,w_\gamma)=\omega_1$.
\end{proof}

For countable $\delta$, put $\mathbf r_G:=(r_\gamma:\gamma<\delta)$,
with $r_\gamma=\omega$ at inactive steps. We call $\mathbf r_G$ the
\emph{history} of $G$; using a fixed constructible coding, we also regard
it as a real. The history determines the extension:
$W[G]=W[\mathbf r_G]$. Indeed, if $H$ is $\AD(Y)$-generic with generic
real $r$, then $H=H_r^Y$, where
\begin{equation}\label{eq:recover-ad-filter}
H_r^Y=\{(s,F):s\subseteq r\ \land\
 \forall\zeta\in F\cap Y\ [(r\setminus s)\cap d_\zeta=\varnothing]\}.
\end{equation}
Every pair on the right is compatible with every member of $H$, and
hence belongs to $H$ by genericity. For a restricted factor
$\AD(Y)\res q$ the generic filter is $H_r^Y\cap(\AD(Y)\res q)$, and the
generic filter for $\AD(\varnothing)$ is all of $\AD(\varnothing)$. By
recursion on $\gamma$, $G$ can therefore be computed from $\mathbf r_G$
and the iteration $\PP$.

In the next lemma we work below a condition, since the underlying set of
$\AD(Y)\res q$, unlike that of $\AD(Y)$, can depend on $Y$.

\begin{lemma}\label{lem:reconstruction}
Let $\PP$ be a coding iteration over $W$ of countable length $\delta$,
and let $W'=W\res V$, $V\in L\cap[\omega_1]^{\leq\omega}$, or
$W'=W^{(\xi)}$, $\xi<\omega_1$. Abbreviate
\[
E_{W'}(\eta,w)\iff
\begin{cases}
\eta\in V,&W'=W\res V,\\
\xi\notin I(\eta,w),&W'=W^{(\xi)}.
\end{cases}
\]
Suppose $p^*\forces\forall\gamma<\delta\,
[a_\gamma=1\to E_{W'}(\eta_\gamma,w_\gamma)]$.
For every determined $p\leq p^*$ there are a coding iteration
$\PP^{p\prime}$ over $W'$ of length $\delta$ and sets $C_\gamma$ of
determined conditions, $\gamma\leq\delta$, such that $C_\gamma$ is dense
in both $\PP_\gamma\res(p\res\gamma)$ and $\PP^{p\prime}_\gamma$,
\[
(C_\gamma,\leq_{\PP_\gamma\res(p\res\gamma)})
=(C_\gamma,\leq_{\PP^{p\prime}_\gamma}),
\]
and corresponding generic filters for the two iterations yield the same
values of $w_\gamma,\eta_\gamma,a_\gamma,q_\gamma$ and the same generic
reals. If $p\in G$ and $J$ is the $\PP^{p\prime}_\delta$-generic filter
generated by $G\cap C_\delta$, then
$W'[J]=W'[\mathbf r_G]$ and $W[G]=W[J]=W[\mathbf r_G]$.
Moreover:
\begin{enumerate}[label=(\arabic*)]
\item $W'=W\res V\ \Longrightarrow\
W[G]=W'[\mathbf r_G][G^1\res(\omega_1\setminus V)]$, with the remaining
Cohen forcing $\sigma$-distributive over $W'[\mathbf r_G]$.
\item $W'=W^{(\xi)}\ \Longrightarrow\
W'[\mathbf r_G]\models``S_\xi\text{ is Suslin}''$ and
$W[G]=W'[\mathbf r_G][b_\xi]$.
\end{enumerate}
In either case $W'[\mathbf r_G]$ and $W[G]$ have the same reals and the
same countable sets of ordinals.
\end{lemma}

\begin{proof}
Fix a determined $p\leq p^*$, and let $\widehat\PP$ be the coding
iteration $\PP_\delta\res p$, so that its prefix $\widehat\PP_\gamma$ of
length $\gamma$ is $\PP_\gamma\res(p\res\gamma)$; its restricting pair at
$\gamma\in\supp(p)$ is $p(\gamma)$, and at the other coordinates it is
the original one. Let $C_\gamma$ be the set of determined conditions $t$
of $\widehat\PP_\gamma$ with $\supp(t)\supseteq\supp(p)\cap\gamma$, where
a coordinate $\gamma'\in\supp(p)\cap\gamma$ outside the support of a
condition is given the value $p(\gamma')$. Then $C_\gamma$ is dense in
$\widehat\PP_\gamma$.

Choose nice names with conditions in $C_\gamma$ for the data
$\dot w_\gamma,\dot\eta_\gamma,\dot a_\gamma,\dot q_\gamma$: a name for a
real is a countable subset of $\omega\times C_\gamma$, and any other of
these names is a countable antichain in $C_\gamma$ together with a value
assigned to each of its elements. The conditions in $C_\gamma$ are finite
functions with values in $[\omega]^{<\omega}\times[\omega_1]^{<\omega}$,
so the countable sequence of all these names belongs to $W'$ by
Lemma~\ref{lem:W}(b),(c).

In $W'$ we define $\widehat\PP'$ by recursion on $\gamma$. The names are
evaluated with respect to $\widehat\PP'_\gamma$: a name for a real is
evaluated bit by bit, and a name given by an antichain takes the value
assigned to the unique element of the antichain in the generic filter,
or a fixed default value if there is no such unique element. Elements of
$C_\gamma$ which are not conditions of $\widehat\PP'_\gamma$ are ignored.
In this way every name has a value in every generic extension. At step
$\gamma$, let $Y'_\gamma:=Y(\eta_\gamma,w_\gamma)$ if $a_\gamma=1$, the
area $\eta_\gamma$ differs from the areas of the earlier active steps, and
$E_{W'}(\eta_\gamma,w_\gamma)$ holds; otherwise let $a'_\gamma:=0$ and
$Y'_\gamma:=\varnothing$. The $\gamma$-th factor of $\widehat\PP'$ is
$\AD(Y'_\gamma)$ restricted to the pair $q_\gamma$, and we put
$\PP^{p\prime}:=\widehat\PP'$. By the condition
$E_{W'}(\eta_\gamma,w_\gamma)$, the Cohen subset $g_{\eta_\gamma}$ and the
branches needed to compute $\Code(w_\gamma,\eta_\gamma)$ belong to $W'$,
so $\PP^{p\prime}$ is a coding iteration over $W'$.

Let $C'_\gamma$ be defined from $\widehat\PP'$ as $C_\gamma$ was defined
from $\widehat\PP$. We show by induction on $\gamma\leq\delta$:
\[
(\ast)_\gamma:\quad C'_\gamma=C_\gamma,\qquad
{\leq_{\widehat\PP'_\gamma}}\res C_\gamma
={\leq_{\widehat\PP_\gamma}}\res C_\gamma,
\]
and the names have the same values in both iterations. The case
$\gamma=0$ is immediate, and limit steps follow since the iterations have
finite support.

Assume $(\ast)_\gamma$, and let $H$ be $\widehat\PP_\gamma$-generic over
$W$. Then $H\cap C_\gamma$ generates a $\widehat\PP'_\gamma$-generic
filter $H'$ over $W'$, since every subset of $C_\gamma$ in $W'$ which is
dense in $\widehat\PP'_\gamma$ is also dense in $\widehat\PP_\gamma$.
Hence the names have the same values under $H$ and $H'$. The condition
$p\res\gamma$ forces that the names have values of the right type, that
$\eta_\gamma$ differs from the earlier active areas if $a_\gamma=1$, and
that $a_\gamma=1\to E_{W'}(\eta_\gamma,w_\gamma)$. These statements also
hold under $H'$: otherwise some condition of $\widehat\PP'_\gamma$ forces
their failure, and strengthening it to a condition in $C_\gamma$ and
taking a $\widehat\PP_\gamma$-generic filter over $W$ which contains this
condition gives a contradiction. Thus the default values are never used,
and since $Y$ is computed by the same absolute formula from $w_\gamma$,
$g_{\eta_\gamma}$ and the selected branches, we get
$Y_\gamma^H=(Y'_\gamma)^{H'}$.

Consequently, for $t\in C_\gamma$, $\zeta<\omega_1$ and
$e\in[\omega]^{<\omega}\times[\omega_1]^{<\omega}$,
\begin{align*}
t\forces_{\widehat\PP_\gamma}\zeta\notin\dot Y_\gamma
&\iff t\forces_{\widehat\PP'_\gamma}\zeta\notin\dot Y'_\gamma,\\
t\forces_{\widehat\PP_\gamma}\dot q_\gamma=\check e
&\iff t\forces_{\widehat\PP'_\gamma}\dot q'_\gamma=\check e.
\end{align*}
(If $t$ forces the left-hand side but not the right-hand side, strengthen
$t$ within $C_\gamma$ to a condition forcing the negation of the
right-hand side, and take a $\widehat\PP_\gamma$-generic filter over $W$
which contains that condition.) For pairs $f=(t,H)$ and $e=(s,F)$ we have
\[
f\leq_Y e\iff s\subseteq t\ \land\ F\subseteq H\ \land\
\bigwedge_{\substack{\zeta\in F\\(t\setminus s)\cap d_\zeta\ne\varnothing}}
\zeta\notin Y,
\]
a condition on finitely many $\zeta$, which determines both the order of
$\AD(Y)$ and membership in $\AD(Y)\res q$. The displayed equivalences
therefore yield $(\ast)_{\gamma+1}$.

By $(\ast)_\delta$, the filter $J$ is $\PP^{p\prime}_\delta$-generic over
$W$, and hence over $W'$. The histories of $G$ and $J$ agree, and
\eqref{eq:recover-ad-filter} gives $W'[J]=W'[\mathbf r_G]$ and
$W[G]=W[J]$. In case (1), let $H:=G^1\res(\omega_1\setminus V)$. By
Lemma~\ref{lem:W}(b), $W=W'[H]$; as $J$ is generic over $W'[H]$, the
product lemma shows that $H$ is generic over $W'[J]$, and its forcing is
$\sigma$-distributive there by Lemma~\ref{lem:W}(b), since $W'[J]$ is a
ccc extension of $W'$. In case (2), $\widehat\PP'$ is a finite support
iteration of $\sigma$-centered forcings, so $S_\xi$ remains Suslin in
$W'[J]$ by Fact~\ref{fact:suslin-preservation}(c). By the product lemma,
$b_\xi$ is $S_\xi$-generic over $W'[J]$; as $S_\xi$ is a normal Suslin
tree in $W'[J]$, it is $\omega$-distributive there, so $W'[J]$ and $W[G]$
have the same countable sequences of ordinals.
\end{proof}

\begin{definition}\label{def:Vpart}
Let $G$ be generic for a countable coding iteration $\PP$ over $W$, and
let $V\in L\cap[\omega_1]^{\leq\omega}$ contain its active areas. Put
\[
W[G]\res V:=(W\res V)[\mathbf r_G].
\]
\end{definition}

Choose $p^*\in G$ forcing that all active areas belong to $V$, and a
determined $p\in G$ with $p\leq p^*$. Lemma~\ref{lem:reconstruction}(1)
gives
\[
W[G]=(W[G]\res V)[G^1\res(\omega_1\setminus V)],
\qquad
\R^{W[G]\res V}=\R^{W[G]},
\]
and the two models have the same countable sets of ordinals.

The model $W[G]\res V$ depends only on $W[G]$ and $V$, not on the
history: if $M=W[h]=W[h']$, where $h,h'$ are histories of coding
iterations whose active areas lie in $V$, put $N:=(W\res V)[h]$ and
$N':=(W\res V)[h']$. Then $\R^N=\R^{N'}=\R^M$, and since histories are
coded by reals, $h'\in N$ and $h\in N'$; hence $N=N'$.

\begin{lemma}\label{lem:coding}
Let $G$ be generic for a coding iteration $\PP$ over $W$ of length
$\delta\leq\omega_1$, and put $A:=\{\gamma<\delta:a_\gamma=1\}$.
Then
\[
W[G]\models\forall w\in2^\omega\,
 [\Phi(w)\iff\exists\gamma\in A\ (w=w_\gamma)].
\]
\end{lemma}

\begin{proof}
($\Leftarrow$) For $\gamma\in A$, Lemma~\ref{lem:coding-iteration-basic}(d)
and Fact~\ref{fact:localization} give
\[
Y(\eta_\gamma,w_\gamma)=\{\zeta:|r_\gamma\cap d_\zeta|<\aleph_0\}
\ \Longrightarrow\ W[G_{\gamma+1}]\models\Psi(r_\gamma,w_\gamma).
\]
Since $\Psi\in\Pi^1_2$, Shoenfield absoluteness gives
$W[G]\models\Psi(r_\gamma,w_\gamma)$.

($\Rightarrow$) Suppose $W[G]\models\Psi(r,w)$. If $\delta=\omega_1$,
Lemma~\ref{lem:coding-iteration-basic}(c) and Shoenfield absoluteness give
some $\gamma_0<\omega_1$ with $r,w\in W[G_{\gamma_0}]$ and
$W[G_{\gamma_0}]\models\Psi(r,w)$.
It therefore suffices to treat $\delta<\omega_1$. Assume toward a
contradiction that $w\notin\{w_\gamma:\gamma\in A\}$.

\emph{Step 1: Decoding.} Put $\kappa:=\omega_2^{W[G]}$.
As $\omega_1^{L[r\oplus w]}=\omega_1=\omega_1^L$, the transitive collapse
$\bar M$ of a countable elementary submodel of $L_\kappa[r\oplus w]$
containing $r,w$ belongs to $\mathcal M(r,w)$. By $\Psi(r,w)$,
$\bar M\models\exists N\in L[r]\ \Theta(N,\omega_1,w)$, so by elementarity
\[
L_\kappa[r\oplus w]\models\exists N\in L[r]\ \Theta(N,\omega_1,w).
\]
Take such $N\in L_\kappa[r]$ and $h\in N$ witnessing clause (ii).
Since $N$ is transitive with $\omega_1^N=\omega_1$, the trees computed in
$N$ are the trees $S_\xi$ by Fact~\ref{fact:trees}; hence
\[
\sup h=\omega_1,\qquad
\forall\gamma\in h\ \forall n<\omega\quad
L[r]\models``S_{\omega\gamma+2n+w(n)}\text{ has a cofinal branch}''.
\]

\emph{Step 2: An unused tree.} Since the active areas are pairwise
distinct and $A$ is countable, Lemma~\ref{lem:areas} gives $\nu_0<\omega_1$
such that
\[
h_{\eta_\gamma}\cap h_{\eta_{\gamma'}}\subseteq\nu_0
\quad\text{for all distinct }\gamma,\gamma'\in A.
\]
Choose $\gamma^*\in h$ and $n<\omega$ as follows. If
$h\cap h_{\eta_\gamma}$ is bounded in $\omega_1$ for every $\gamma\in A$,
take $\gamma^*\in h$ above all these bounds, so that
$\gamma^*\notin\bigcup_{\gamma\in A}h_{\eta_\gamma}$, and let $n$ be
arbitrary. Otherwise choose $\gamma_1\in A$ with
$\sup(h\cap h_{\eta_{\gamma_1}})=\omega_1$, then
$\gamma^*\in(h\cap h_{\eta_{\gamma_1}})\setminus\nu_0$, and $n$ with
$w(n)\ne w_{\gamma_1}(n)$. Then $\gamma^*\notin h_{\eta_\gamma}$ for
$\gamma\in A\setminus\{\gamma_1\}$. In either case, since the
representation of an ordinal as $\omega\gamma+2n+j$ with $j<2$ is unique,
we get
\[
\xi:=\omega\gamma^*+2n+w(n)
\ \notin\ \bigcup_{\gamma\in A}I(\eta_\gamma,w_\gamma).
\]

\emph{Step 3: Omitting $b_\xi$.} Choose $p^*\in G$ with
$p^*\forces\forall\gamma<\delta\,[a_\gamma=1\to
\xi\notin I(\eta_\gamma,w_\gamma)]$, then determined $p\in G$, $p\leq p^*$.
By Lemma~\ref{lem:reconstruction}(2),
\[
r\in W^{(\xi)}[\mathbf r_G],\qquad
W^{(\xi)}[\mathbf r_G]\models``S_\xi\text{ is Suslin}''.
\]
Hence $L[r]\subseteq W^{(\xi)}[\mathbf r_G]$ contains no cofinal branch
through $S_\xi$, contradicting Step~1.
\end{proof}

\section{Allowable forcings and end-extension requirements}\label{sec:allowable}

\subsection{Stages and ordered coding iterations}

\begin{definition}\label{def:stage}
A \emph{stage} is a quadruple $s=(M,V,\prec,E)$ with the following
properties.
\begin{enumerate}[label=(St\arabic*)]
\item $M=W[h]$, where $h=(r_\gamma:\gamma<\delta_s)$, $\delta_s<\omega_1$,
is the history of a generic filter for a coding iteration over $W$. The
history $h$ is part of the data.
\item $V\in L\cap[\omega_1]^{\leq\omega}$ and
$\{\eta_\gamma:\gamma<\delta_s,\ a_\gamma=1\}\subseteq V$.
The elements of $V$ are called \emph{unavailable}, those of
$\omega_1\setminus V$ \emph{available}. Put
\[
N_s:=M\res V=(W\res V)[h]=L[\vec b,\vec g\res V,h],
\]
where $\vec b=(b_\xi:\xi<\omega_1)$ and
$\vec g\res V=(g_\eta:\eta\in V)$.
\item $\prec\in N_s$ and $(\R^M,\prec)$ is well-ordered.
\item $E=(\epsilon_\nu:\nu<\lh(E))\in N_s$, $\lh(E)<\omega_1$, where
\[
\epsilon_\nu\in\mathcal P(\R^M)\ \cup\ ((\R^M)^2\times\omega^2).
\]
Entries of the first kind are called \emph{end-extension requirements},
those of the second kind \emph{separation requirements}; distinct separation
requirements in $E$ have distinct first coordinates. We call $E$ the
\emph{requirement list} of $s$.
\end{enumerate}
Since $M=W[h]$, the stage is determined by the set
$d_s:=(h,V,\prec,E)$.
\end{definition}

Lemma~\ref{lem:reconstruction} gives
\[
\begin{gathered}
M=N_s[H_V],\qquad H_V:=G^1\res(\omega_1\setminus V),\\
N_s\models``\Cohen(\omega_1\setminus V)
\text{ is $\sigma$-distributive}''.
\end{gathered}
\]
Consequently $\R^M=\R^{N_s}$ and $(\,{}^\omega N_s)^M\subseteq N_s$;
in particular the two models have the same countable sets of ordinals.
Clearly $h,V\in N_s$, and $\prec,E\in N_s$ by (St3)--(St4); as
$N_s\models\CH$, we get $d_s\in H(\omega_2)^{N_s}$.

For histories and unavailable sets as above, the restricted models are
monotone:
\[
h\subseteq h',\quad V\subseteq U
\quad\Longrightarrow\quad
L[\vec b,\vec g\res V,h]\subseteq L[\vec b,\vec g\res U,h'].
\]
This holds since the sets generating the left-hand model belong to the
right-hand model. The same argument applies to the factors of finite
products and their prefixes, whose histories can be read off from the
concatenated history.

\begin{definition}\label{def:ordered}
Let $s=(M,V,\prec,E)$ be a stage. An \emph{ordered coding iteration over
$s$} consists of a coding iteration
$\QQ=(\QQ_\gamma,\dot\QQ(\gamma):\gamma<\delta)$ over $M$,
$\delta<\omega_1$, with its names
$(\dot w_\gamma,\dot\eta_\gamma,\dot a_\gamma,\dot q_\gamma)$,
and $\QQ_\gamma$-names $(\dot V_\gamma,\dot\prec_\gamma)$,
$\gamma\leq\delta$, such that $\QQ_\delta$ forces the following.
\begin{enumerate}[label=(O\arabic*)]
\item $V_0=V$, $V_\gamma\in L\cap[\omega_1]^{\leq\omega}$, and
$\gamma\leq\gamma'\leq\delta\Rightarrow V_\gamma\subseteq V_{\gamma'}$.
\item $\gamma<\delta\ \wedge\ a_\gamma=1
\Rightarrow\eta_\gamma\in V_{\gamma+1}\setminus V_\gamma$.
\item $\prec_0=\prec$ and, for $\gamma\leq\gamma'\leq\delta$,
\[
\prec_\gamma\in M[G_\gamma]\res V_\gamma,\qquad
(\R^{M[G_\gamma]},\prec_\gamma)\text{ is well-ordered},\qquad
\prec_{\gamma'}\res\R^{M[G_\gamma]}=\prec_\gamma.
\]
\end{enumerate}
\end{definition}

By (O1)--(O2), the active areas are distinct and outside $V$. Hence
$h_\gamma:=h{}^\frown\mathbf r_{G_\gamma}$ is the history of a generic
filter for a coding iteration over $W$, namely the iteration giving $h$
followed by $\QQ_\gamma$, and all its active areas lie in $V_\gamma$.
Consequently
\[
s_\gamma:=(M[G_\gamma],V_\gamma,\prec_\gamma,E)
\quad(\gamma\leq\delta),\qquad s[G]:=s_\delta,
\]
are stages, called the \emph{stages of $\QQ$} and its \emph{terminal
stage}, respectively. Lemma~\ref{lem:coding} applies to $M[G]$. Clause (O3)
still allows new reals to be placed below old reals; end-extension
requirements exclude this:

\begin{definition}\label{def:frozen}
Let $F\in\Init(\prec)$, where $(\R^M,\prec)$ is well-ordered.
An ordered coding iteration over a stage with order $\prec$
\emph{respects the end-extension requirement $F$}, or \emph{freezes
the well-order on $F$}, if
\[
\QQ_\delta\forces\ \forall\gamma\leq\delta\,
  \bigl(F\in\Init(\prec_\gamma)\bigr).
\]
\end{definition}

By (O3), freezing gives $\prec_\gamma\res F=\prec\res F$ and
$u\prec_\gamma r$ whenever $u\in F$ and
$r\in\R^{M[G_\gamma]}\setminus\R^M$.

\subsection{Allowable forcings}

Fix a recursive enumeration $(\varphi_m(x,y):m<\omega)$ of the
$\Sigma^1_3$ formulas in two free real variables, and put
$A_m(y):=\{x\in\R:\varphi_m(x,y)\}$, evaluated in the ambient model.
Every pair of $\Sig^1_3$ sets has the form $(A_m(y),A_k(y))$, with $y$
coding the parameters. For a requirement list $E$ write
$\mathcal F_\alpha(E):=\{\epsilon_\nu:\nu<\alpha,\ \epsilon_\nu
\text{ is an end-extension requirement}\}$.

\begin{definition}\label{def:allowable}
For a stage $s=(M,V,\prec,E)$ and $\alpha\leq\lh(E)$, let
$\mathcal C_\alpha(s)$ denote the class of \emph{$\alpha$-allowable}
ordered coding iterations over $s$. It is defined by recursion on
$\alpha$, simultaneously for all stages. Once $\mathcal C_\nu$ has been
defined, put
\[
\Poss_\nu(s,\varphi)\quad:\Longleftrightarrow\quad
\exists\QQ\in\mathcal C_\nu(s)\ (\QQ\forces\varphi)
\]
for $\Sigma^1_3$ sentences $\varphi$ with real parameters from $M$; we
call these statements \emph{tests}. If $\epsilon_\nu=(c,y,m,k)$ and
$x\in\R^M$, define $\side_\nu(s,x)\in\{0,1\}$ by
\[
\side_\nu(s,x)=1\quad\Longleftrightarrow\quad
\Poss_\nu(s,x\in A_k(y))\ \wedge\
\neg\Poss_\nu(s,x\in A_m(y)).
\]
Thus $0$ has priority and is the default value when neither test holds.

An ordered coding iteration $\QQ$ over $s$ belongs to
$\mathcal C_\alpha(s)$ iff $\QQ_\delta$ forces the following for all
$\gamma\leq\delta$, all active $\gamma'<\delta$ and all $\nu<\alpha$:
\begin{enumerate}[label=(A\arabic*)]
\item if $w_{\gamma'}=\la\WO,u,v\ra$, then $u\neq v$ and
$u\prec_{\gamma'}v$;
\item $\mathcal F_\alpha(E)\subseteq\Init(\prec_\gamma)$;
\item if $\epsilon_\nu=(c,y,m,k)$ and
$w_{\gamma'}=\la\SEP,c,y,m,k,x,i\ra$, then $i=\side_\nu(s_{\gamma'},x)$.
\end{enumerate}
\end{definition}

Here ``$\QQ_\delta$ forces'' refers to the trivial condition; since
restrictions are again coding iterations, this is no loss of generality.
In (A3) the side is evaluated in $M[G_{\gamma'}]$, and only indices
$\nu<\alpha$ occur. Thus $\mathcal C_\alpha(s)$ depends only on
$(M,V,\prec,E\res\alpha)$, and $\side_\nu(s,x)$ only on
$(M,V,\prec,E\res(\nu+1))$. Clause (A2) says that $\QQ$ freezes the
well-order on every $F\in\mathcal F_\alpha(E)$. The codes $\la\AUX\ra$, and
the codes $\la\SEP,c,y,m,k,x,i\ra$ with
$(c,y,m,k)\notin\operatorname{ran}(E\res\alpha)$, are unrestricted, and
membership in $\mathcal C_0(s)$ only requires (A1).

\begin{lemma}\label{lem:recursion}
The classes $\mathcal C_\alpha$ and the tests $\Poss_\alpha$ are
definable from $\vec b$ and $\vec g$, uniformly in the stage parameters
$d_s$. Moreover, if $G$ is generic over $M$ for a countable coding
iteration, then these classes and tests as computed in $M[G]$ are
obtained by evaluating at $G$ Boolean values computed in $M$.
\end{lemma}

\begin{proof}
We represent forcings by dense sets of determined conditions. Let
$B_0:=[\omega]^{<\omega}\times[\omega_1]^{<\omega}$. For a countable
coding iteration $\PP$, the set $D_\PP$ is dense in $\PP_{\lh(\PP)}$,
and we regard it as a subset of $\operatorname{Fn}(\lh(\PP),B_0;{<}\omega)$,
the set of finite partial functions from $\lh(\PP)$ to $B_0$. The names
for the data of the factors are taken to be nice names with conditions in
$D_\PP$. By the ccc and $\CH$ there are only $\aleph_1$ nice names for
reals; likewise, names for unavailable sets, for well-orders of the reals
and for requirement lists can be chosen of hereditary size
$\leq\aleph_1$, using countable antichains which decide the relevant
statements. We call such a representation a \emph{normalized
presentation}. Normalized presentations belong to $H(\omega_2)$, and
stages are represented by their parameters $d_s$.

Fix $M$. Let $\mathcal K^M\subseteq H(\omega_2)^M$ be the set of
normalized presentations of countable coding iterations over $M$, and
put $\mathbb B_\PP:=\operatorname{RO}(\PP)$ for $\PP\in\mathcal K^M$.
Stages and further iterations in generic extensions are named by
elements of $H(\omega_2)^M$:
\[
u\in H(\omega_2)^{M[G]}
\quad\Longrightarrow\quad
\exists\dot u\in H(\omega_2)^M\ (\dot u^G=u).
\tag{$\dagger$}
\]
Indeed, $u$ is coded by a relation on $\omega_1$ which is isomorphic to
the membership relation on the transitive closure of $\{u\}$, and by the
ccc this relation has a name built from countable antichains. Hence
\[
I^M:=\{(\PP,\dot d,\dot\QQ):\PP\in\mathcal K^M,
\ \dot d,\dot\QQ\in H(\omega_2)^M\text{ are $\PP$-names}\},
\]
where $\dot d$ names a stage parameter and $\dot\QQ$ an iteration, and
the algebras $\mathbb B_\PP$ are sets in $M$. Candidate witnesses for
the tests are named by normalized presentations of coding iterations. If
a candidate is given only below a condition $p$, we mix it with the
trivial iteration of length $0$ on the complement of the Boolean value
of $p$. Hence the composition of $\PP$ with a prefix of a candidate is a
coding iteration below every condition; to names which do not denote
stages or well-orders we assign the Boolean value $0$.

By recursion on $\alpha<\omega_1$ we define Boolean values
\[
b_\alpha(\PP,\dot d,\dot\QQ)\in\mathbb B_\PP,
\qquad (\PP,\dot d,\dot\QQ)\in I^M,
\]
which are intended to be the Boolean values of
$\dot\QQ\in\mathcal C_\alpha(\dot d)$. Clauses (O1)--(O3), (A1) and (A2)
are expressed by forcing statements over $\PP$. For (A3), we append the
relevant prefix of $\dot\QQ$ to $\PP$; by the ccc, its length is forced
to be below a countable ordinal of $M$, and padding with inactive steps
(Lemma~\ref{lem:padding}) gives a normalized presentation
$\PP'\in\mathcal K^M$ of the composition. At each such prefix, the
Boolean value of the test for $\nu<\alpha$ is
\[
\bigvee_{\dot{\mathbb U}}\left(
 b_\nu(\PP',\dot d',\dot{\mathbb U})\wedge
 \|\dot{\mathbb U}\forces\varphi\|_{\mathbb B_{\PP'}}\right),
\]
where $\dot d'$ is the name for the stage parameter at the prefix, and
$\dot{\mathbb U}$ ranges over the $\PP'$-names for candidate witnesses.
By Definition~\ref{def:allowable}, the values of the two tests determine
the Boolean value of the side. Since only the values $b_\nu$ with
$\nu<\alpha$ are used, this is a recursion on $\alpha$ in $M$.

We show by induction on $\alpha$ that these values are correct, i.e.,
that for every $\PP$-generic $G$ over $M$, the statement
$\dot\QQ^G\in\mathcal C_\alpha(\dot d^G)$ holds in $M[G]$ iff
$b_\alpha(\PP,\dot d,\dot\QQ)\in G$. For $\alpha=0$, this is the forcing
theorem applied to (O1)--(O3) and (A1). For $\alpha>0$, every candidate
witness in $M[G]$ has a normalized name by $(\dagger)$, and by the
induction hypothesis its $\nu$-allowability and the sides for $\nu<\alpha$
are given by the values $b_\nu$. The forcing theorem then yields the
correctness of the tests and of (A3). This proves the lemma.
\end{proof}

\begin{lemma}\label{lem:basic}
Let $s=(M,V,\prec,E)$ be a stage and $\alpha\leq\lh(E)$.
\begin{enumerate}[label=(\alph*)]
\item $\alpha'\leq\alpha
\Rightarrow\mathcal C_\alpha(s)\subseteq\mathcal C_{\alpha'}(s)$.
\item If $\mathcal F_\alpha(E)\subseteq\Init(\prec)$, then
$\mathbf1\in\mathcal C_\alpha(s)$, where $\mathbf1$ is the trivial
ordered coding iteration of length $0$.
\item If
$\QQ\in\mathcal C_\alpha(s)$ and $q\in\QQ_\delta$, then
$\QQ\res q\in\mathcal C_\alpha(s)$.
\item If $\QQ\in\mathcal C_\alpha(s)$ and
$\gamma_0\leq\delta$, then
\[
\QQ\res\gamma_0\in\mathcal C_\alpha(s),\qquad
\QQ_{\gamma_0}\forces
\QQ_{[\gamma_0,\delta)}\in\mathcal C_\alpha(s_{\gamma_0}).
\]
The tail carries the names $(\dot V_\gamma,\dot\prec_\gamma)$,
$\gamma_0\leq\gamma\leq\delta$.
\item If $\TT\in\mathcal C_\alpha(s)$ and
$\TT\forces\dot\QQ\in\mathcal C_\alpha(s[\dot G])$, then
$\TT*\dot\QQ\in\mathcal C_\alpha(s)$, with concatenated names and,
if necessary, padding of $\dot\QQ$ to a fixed countable length.
\item Suppose
$\mathcal F_\alpha(E)\subseteq\Init(\prec)$,
$w\in\R^M$, $\eta\in\omega_1\setminus V$, and
\begin{align*}
w=\la\WO,u,v\ra&\ \Rightarrow\ u\neq v\ \wedge\ u\prec v,\\
w=\la\SEP,c,y,m,k,x,i\ra,\quad\epsilon_\nu=(c,y,m,k),\quad\nu<\alpha
&\ \Rightarrow\ i=\side_\nu(s,x).
\end{align*}
Let $\QQ$ be the one-step coding iteration $\Code(w,\eta)$, with
$V_1:=V\cup\{\eta\}$, and let $\dot\prec_1$ be a name for a well-order of
$\R^{M[G]}$ with $\prec_1\in M[G]\res V_1$ and $\prec\unlhd\prec_1$ (for
instance, $\prec$ followed by a well-order of the new reals which belongs
to $M[G]\res V_1$). Then $\QQ\in\mathcal C_\alpha(s)$ and
$\R^M\in\Init(\prec_1)$.
\end{enumerate}
\end{lemma}

\begin{proof}
(a) holds since membership in $\mathcal C_{\alpha'}(s)$ only involves the
requirements in $E\res\alpha'$. For (b), (A1) and (A3) are vacuous and
(A2) is the hypothesis. For (c)--(e), use the canonically lifted names:
the stages of a restriction, prefix, tail or composition are the
corresponding stages of the original forcings, for generic filters which
contain the relevant conditions. Thus (O1)--(O3) and (A1)--(A3) are
inherited stage by stage; in (d) the stage of the tail corresponding to
$\gamma$ is $s_\gamma$.

For (f), (O1)--(O2) hold, and
\[
\prec\in M\res V\subseteq M[G]\res V_1,\qquad
\R^{M[G]}=\R^{M[G]\res V_1},\qquad\prec\unlhd\prec_1
\]
show that $\prec_1$ can be chosen as required and give (O3).
At the only active step, $s_0=s$; the two displayed implications give
(A1) and (A3). Finally
$\mathcal F_\alpha(E)\subseteq\Init(\prec)$ and
$\prec\unlhd\prec_1$ give (A2) at both stages.
\end{proof}

\begin{lemma}\label{lem:padding}
For $\QQ\in\mathcal C_\alpha(s)$ of length $\delta$ and
$\delta\leq d<\omega_1$ there is
$\operatorname{pad}_d(\QQ)\in\mathcal C_\alpha(s)$ such that
\[
\lh(\operatorname{pad}_d(\QQ))=d,\qquad
\operatorname{pad}_d(\QQ)\res\delta=\QQ,
\]
and the padded iteration is forcing equivalent to $\QQ$, with the
same terminal model, unavailable set and order.
Moreover, for a countable coding iteration $\TT$ and a $\TT$-name
$\dot\QQ$ for an iteration,
\[
\TT\forces\lh(\dot\QQ)<\omega_1
\quad\Longrightarrow\quad
\exists d<\omega_1\ (\TT\forces\lh(\dot\QQ)<\check d).
\]
\end{lemma}

\begin{proof}
For $\delta\leq\gamma<d$ use
\[
a_\gamma=0,\quad w_\gamma=\la\AUX\ra,\quad\eta_\gamma=0,\quad
q_\gamma=(\varnothing,\varnothing),\quad
(V_\gamma,\prec_\gamma)=(V_\delta,\prec_\delta).
\]
The factors $\AD(\varnothing)$ are directed and add no new sets, and the
history is only extended by entries $\omega$. Hence neither the models
$M[G_\gamma]$ nor the models $M[G_\gamma]\res V_\gamma$ change. The new
instances of (A2) coincide with the instance at $\delta$, and since the
new steps are inactive, (A1) and (A3) impose no new conditions. For the
bound, a maximal antichain deciding $\lh(\dot\QQ)$ is countable by the
ccc; take $d$ strictly above its possible values. This also justifies the
padding in Lemma~\ref{lem:basic}(e): the padded composition has a fixed
countable length and satisfies the same conditions at each stage.
\end{proof}

\subsection{Relocation}\label{sec:relocation}

Witnesses for tests use available coding areas. To compare tests over
different models, we move witnesses away from prescribed countable sets
of areas, using the weak homogeneity of the Cohen product; compare
\cite[proof of Lemma~5.1]{HoFailRed} and
\cite[Lemmas~6.12 and~6.13]{HoSigma11}.

\begin{lemma}\label{lem:relocation}
Let $s=(M,V,\prec,E)$ be a stage, $\alpha\leq\lh(E)$, and let $\varphi$ be a
$\Sigma^1_3$ sentence with real parameters from $M$. If
$\QQ\in\mathcal C_\alpha(s)$ and $\QQ\forces\varphi$, then for every
countable $B\subseteq\omega_1$ there are $\QQ'\in\mathcal C_\alpha(s)$
and $Z\in L\cap[\omega_1]^{\leq\omega}$ with
\[
Z\cap(V\cup B)=\varnothing,\qquad
\QQ'\forces(\varphi\ \wedge\ V^{\QQ'}_{\delta'}\subseteq V\cup Z).
\]
\end{lemma}

\begin{proof}
Put $N:=N_s$, $\Cohen:=\Cohen(\omega_1\setminus V)$ and $H:=H_V$.
Then $M=N[H]$ and
\[
d_s=(h,V,\prec,E),\quad\alpha,\quad\vec b,
\quad\text{the real parameters of }\varphi\ \in N.
\]
Here $\prec,E\in N$ by (St3)--(St4), and the real parameters belong to $N$
since $\R^M=\R^N$. These parameters have check names, and the sequence
$(g_\eta:\eta\in\omega_1\setminus V)$ of available Cohen subsets has a
canonical $\Cohen$-name.

For a permutation $\pi\in L$ of $\omega_1$ with $\pi\res V=\mathrm{id}$,
put
\[
(\hat\pi p)(\pi(\eta)):=p(\eta),\qquad
\vec g^{\,\pi}:=(g_{\pi(\eta)}:\eta<\omega_1).
\]
Thus $\hat\pi\in\operatorname{Aut}(\Cohen)$. For a normalized
presentation computed from $\vec g^{\,\pi}$, let $R_\pi$ of it be the
normalized presentation computed from $\vec g$ which is obtained by
replacing every area $\eta$ by $\pi(\eta)$ and every unavailable set $U$
by $\pi[U]$. All other data, namely the forcing orders, histories, reals
and well-orders, remain unchanged. This is well defined, since
\begin{align*}
\Code^{\vec g^{\,\pi}}(w,\eta)
 &=\Code^{\vec g}(w,\pi(\eta)),\\
L[\vec b,\vec g^{\,\pi}\res U,h']
 &=L[\vec b,\vec g\res\pi[U],h'].
\end{align*}
By these identities, $R_\pi$ preserves the requirements on coding
iterations and stages, as well as (O1)--(O3), (A1) and (A2). By induction
on $\nu<\alpha$, $R_\pi$ is a bijection between the $\nu$-allowable
witnesses computed from $\vec g^{\,\pi}$ and those computed from $\vec g$,
with inverse $R_{\pi^{-1}}$; since the forcing orders and the statements
about reals do not change, the tests and the sides agree. Thus $R_\pi$
also preserves (A3). By Lemma~\ref{lem:recursion}, the same induction
applies to names in the relevant generic extensions.

Abbreviate
$\psi(\QQ):\Longleftrightarrow
(\QQ\in\mathcal C_\alpha(s)\ \wedge\ \QQ\forces\varphi)$.
The automorphism lemma and this equivariance yield
\[
p\forces_{\Cohen}\psi(\dot\QQ)
\quad\Longrightarrow\quad
\hat\pi p\forces_{\Cohen}
\psi\bigl(R_\pi(\hat\pi\dot\QQ)\bigr).
\tag{$*$}
\]
Here $\hat\pi\dot\QQ$ names an iteration computed from $\vec g^{\,\pi}$,
and $R_\pi$ turns it into one computed from $\vec g$; the check names for
the parameters are fixed by $\hat\pi$.

Choose a $\Cohen$-name $\dot\QQ$ for the given witness. By the ccc of
$\QQ_\delta$, there are only countably many possible values of
$V^{\QQ}_\delta$, each a countable set in $L$. Their union is a countable
set of ordinals in $M$, hence covered by a countable set in $L$, since
$M$ is a ccc extension of $L[G^1]$, which has the same countable sets of
ordinals as $L$. Removing $V$ from such a cover gives
$Z_0\in L\cap[\omega_1]^{\leq\omega}$ such that
\[
Z_0\cap V=\varnothing,\qquad
M\models\QQ\forces V^{\QQ}_\delta\subseteq V\cup Z_0.
\]
By the forcing theorem choose $p\in H$ forcing both
$\psi(\dot\QQ)$ and this bound with $\check Z_0$.

Fix a countable $B\subseteq\omega_1$ in $M$; by $\sigma$-distributivity,
$B\in N$. Given $q\in\Cohen$, put $Z_1:=Z_0\cup\supp(p)$, and choose
$\zeta<\omega_1$ and $Z_2$ with
\begin{align*}
\zeta&>\sup(V\cup B\cup Z_1\cup\supp(q)),\\
Z_2&\in L,\quad Z_2\subseteq[\zeta,\omega_1),\quad
\operatorname{otp}(Z_2)=\operatorname{otp}(Z_1).
\end{align*}
Let $\pi\in L$ exchange $Z_1$ and $Z_2$ in an order-preserving way and be
the identity elsewhere. Then
\[
\pi\res V=\mathrm{id},\qquad
\supp(\hat\pi p)\subseteq Z_2,\qquad
\supp(\hat\pi p)\cap\supp(q)=\varnothing.
\]
Hence $q$ and $\hat\pi p$ have a common extension $r$. By $(*)$, $r$
forces that $R_\pi(\hat\pi\dot\QQ)$ is a witness whose unavailable sets
are contained in $V\cup Z$, where $Z:=\pi[Z_0]$ satisfies
$Z\cap(V\cup B)=\varnothing$. Thus the conditions which force the
conclusion of the lemma are dense in $\Cohen$, and the genericity of $H$
yields $\QQ'$ and $Z$ in $M$.
\end{proof}

\subsection{Finite products}\label{sec:products}

\begin{definition}
Ordered coding iterations $\QQ^0,\QQ^1$ over $s=(M,V,\prec,E)$, of
lengths $\delta_0,\delta_1$, are \emph{disjoint} if there are countable
$Z_0,Z_1\in L$ such that
\[
Z_0\cap Z_1=Z_0\cap V=Z_1\cap V=\varnothing,\qquad
\QQ^i\forces V^i_{\delta_i}\subseteq V\cup Z_i\quad(i<2).
\]
\end{definition}

By Lemma~\ref{lem:relocation}, a witness to $\Poss_\alpha(s,\varphi)$ can
be chosen disjoint from any prescribed $\TT\in\mathcal C_\alpha(s)$:
choose a countable $Z\in L$ with
$\TT\forces V^{\TT}_{\lh(\TT)}\subseteq V\cup Z$ and apply
Lemma~\ref{lem:relocation} with $B:=Z$.

Fix disjoint ordered coding iterations $\QQ^0,\QQ^1$ over $s$ and a
$\QQ^0_{\delta_0}\times\QQ^1_{\delta_1}$-generic
$G^{(0)}\times G^{(1)}$ over $M$. Put
\[
M_{\beta,\gamma}:=M[G^{(0)}_\beta\times G^{(1)}_\gamma],\qquad
R_{\beta,\gamma}:=\R^{M_{\beta,\gamma}}
\quad(\beta\leq\delta_0,\ \gamma\leq\delta_1).
\]
With the names lifted along the product embeddings,
$\QQ^0_\beta\times\QQ^1_\gamma$ is a coding iteration (first
$\QQ^0_\beta$, then $\QQ^1_\gamma$). By disjointness, its active areas
are pairwise distinct, so Lemma~\ref{lem:coding} applies to every
$M_{\beta,\gamma}$.

\begin{lemma}\label{lem:intersection}
For $\beta,\beta'\leq\delta_0$ and $\gamma,\gamma'\leq\delta_1$,
\[
R_{\beta,\gamma}\cap R_{\beta',\gamma'}
=R_{\min(\beta,\beta'),\min(\gamma,\gamma')}.
\]
\end{lemma}

\begin{proof}
The case $\beta\leq\beta'$, $\gamma\leq\gamma'$ is trivial, so by
symmetry we may assume $\beta\leq\beta'$ and $\gamma'\leq\gamma$. Over
$M':=M_{\beta,\gamma'}$, the tails
$H_0:=G^{(0)}_{[\beta,\beta')}$ and
$H_1:=G^{(1)}_{[\gamma',\gamma)}$ are mutually generic. Hence
\[
R_{\beta,\gamma}\cap R_{\beta',\gamma'}
=\R^{M'[H_1]}\cap\R^{M'[H_0]}=\R^{M'}=R_{\beta,\gamma'},
\]
since two mutually generic extensions of $M'$ have only the reals of $M'$
in common.
\end{proof}

\begin{lemma}\label{lem:amalgamation}
There are names $\prec_{\beta,\gamma}$ such that the product forces,
for
$\beta\leq\beta'\leq\delta_0$ and $\gamma\leq\gamma'\leq\delta_1$:
\begin{enumerate}[label=(\alph*)]
\item $(R_{\beta,\gamma},\prec_{\beta,\gamma})$ is a well-order and
$\prec_{\beta,0}=\prec^0_\beta$, $\prec_{0,\gamma}=\prec^1_\gamma$;
\item $\prec_{\beta',\gamma'}\res R_{\beta,\gamma}=\prec_{\beta,\gamma}$;
\item for every $F\subseteq R_{0,0}$,
$F\in\bigcap_{a\leq\delta_0}\Init(\prec^0_a)
       \cap\bigcap_{b\leq\delta_1}\Init(\prec^1_b)$ implies
$F\in\Init(\prec_{\beta,\gamma})$;
\item $\prec_{\beta,\gamma}\in N_{\beta,\gamma}$, where
$N_{\beta,\gamma}:=M_{\beta,\gamma}\res(V^0_\beta\cup V^1_\gamma)$.
\end{enumerate}
\end{lemma}

\begin{proof}
Enumerate $R_{0,0}$ in the order $\prec$ as $(a_\xi:\xi<\theta)$. In
$M_{\beta,\gamma}$, partition the reals into $R_{0,0}$ and
\[
X^0_\beta:=R_{\beta,0}\setminus R_{0,0},\quad
X^1_\gamma:=R_{0,\gamma}\setminus R_{0,0},\quad
X^{\mathrm{mix}}_{\beta,\gamma}:=R_{\beta,\gamma}
                         \setminus(R_{\beta,0}\cup R_{0,\gamma}).
\]
We call the elements of $R_{0,0}$ \emph{old}, those of $X^0_\beta$
\emph{$0$-new}, those of $X^1_\gamma$ \emph{$1$-new}, and those of
$X^{\mathrm{mix}}_{\beta,\gamma}$ \emph{mixed}.
For $r\in X^0_\beta$ define its cut $c_0(r)\leq\theta$ by
$\{\eta<\theta:a_\eta\prec^0_\beta r\}=c_0(r)$, and similarly define
$c_1(r)$ for $r\in X^1_\gamma$. Let
$C^0_\xi:=\{r\in X^0_\beta:c_0(r)=\xi\}$,
$C^1_\xi:=\{r\in X^1_\gamma:c_1(r)=\xi\}$, and
$A_\xi:=\{a_\xi\}$ for $\xi<\theta$, $A_\theta:=\varnothing$.
Order the non-mixed reals by the well-ordered sum
\[
\sum_{\xi\leq\theta}
 \bigl((C^0_\xi,\prec^0_\beta)+(C^1_\xi,\prec^1_\gamma)
                          +(A_\xi,\prec)\bigr).
\]
Thus the $0$-new reals precede the $1$-new reals within each cut.

For $r\in X^{\mathrm{mix}}_{\beta,\gamma}$ put
\[
a_{\beta,\gamma}(r):=\min\{a\leq\beta:r\in R_{a,\gamma}\},\qquad
b_{\beta,\gamma}(r):=\min\{b\leq\gamma:r\in R_{\beta,b}\}.
\]
If these values are $a,b$, Lemma~\ref{lem:intersection} gives
$r\in R_{a,\gamma}\cap R_{\beta,b}=R_{a,b}$. Let
\[
B_{a,b}:=\{r\in X^{\mathrm{mix}}_{a,b}:
             (a_{a,b}(r),b_{a,b}(r))=(a,b)\}.
\]
Lemma~\ref{lem:intersection} also gives the disjoint decomposition
$X^{\mathrm{mix}}_{\beta,\gamma}=\bigcup_{a\leq\beta,\,b\leq\gamma}B_{a,b}$.
Choose a well-order $\lhd_{a,b}$ of $B_{a,b}$ in $N_{a,b}$, and place
after the non-mixed reals the lexicographically ordered sum
\[
\sum_{(a,b)\in(\beta+1)\times(\gamma+1)}^{\mathrm{lex}}
                    (B_{a,b},\lhd_{a,b}).
\]
This defines $\prec_{\beta,\gamma}$. By the maximum principle choose the
$\QQ^0_a\times\QQ^1_b$-names for $\lhd_{a,b}$ simultaneously in $M$,
then use their canonical lifts to all larger products and products with
prefixes. All summands and index orders are well-ordered. Moreover,
\[
\prec_{\beta,\gamma}\res R_{\beta,0}=\prec^0_\beta,\qquad
\prec_{\beta,\gamma}\res R_{0,\gamma}=\prec^1_\gamma,
\]
since $\prec^0_\beta$ is determined by the cuts of the $0$-new reals and
its restrictions to the sets $C^0_\xi$, and similarly for $\prec^1_\gamma$.
This proves (a).

For (b), by the coherence of the orders of $\QQ^0$ and $\QQ^1$, the cuts
of the $0$-new and $1$-new reals of $R_{\beta,\gamma}$ are the same at
$(\beta',\gamma')$. For $r\in X^{\mathrm{mix}}_{\beta,\gamma}$, the
identities
\[
R_{\beta,\gamma}\cap R_{\beta',0}=R_{\beta,0},\qquad
R_{\beta,\gamma}\cap R_{0,\gamma'}=R_{0,\gamma}
\]
show that $r$ remains mixed at $(\beta',\gamma')$. For $a\leq\beta$,
\[
r\in R_{a,\gamma'}
\iff r\in R_{a,\gamma'}\cap R_{\beta,\gamma}
\iff r\in R_{a,\gamma}.
\]
Consequently $a_{\beta',\gamma'}(r)=a_{\beta,\gamma}(r)$; similarly
$b_{\beta',\gamma'}(r)=b_{\beta,\gamma}(r)$. The fixed $\lhd_{a,b}$
therefore give the required restriction equality. These identities also
justify the disjoint decomposition into the $B_{a,b}$ above.

For (c), the hypotheses on $F$ imply
\[
x\in F,\ r\in X^0_\beta\cup X^1_\gamma
\quad\Longrightarrow\quad x\prec_{\beta,\gamma}r.
\]
On $R_{0,0}$ we have $F\in\Init(\prec)$; all mixed reals come last.
Hence $F\in\Init(\prec_{\beta,\gamma})$.

For (d), the monotonicity of the restricted models gives
$N_{a,b}\subseteq N_{\beta,\gamma}$ for $a\leq\beta$, $b\leq\gamma$.
In particular
$\prec,\prec^0_\beta,\prec^1_\gamma,R_{a,b},\lhd_{a,b}\in N_{\beta,\gamma}$.
The countable families $(R_{a,b})_{a,b}$ and $(\lhd_{a,b})_{a,b}$ belong
to $N_{\beta,\gamma}$ as well, since the remaining Cohen forcing is
$\sigma$-distributive over $N_{\beta,\gamma}$ by
Lemma~\ref{lem:reconstruction}(1). Thus the displayed sums define
$\prec_{\beta,\gamma}$ in $N_{\beta,\gamma}$.
\end{proof}

\begin{definition}\label{def:product}
Using the family $(\prec_{\beta,\gamma})$, let $\QQ^0\otimes\QQ^1$ be the
ordered coding iteration which first performs $\QQ^0$ and then $\QQ^1$,
with unavailable sets and orders
\[
(V_\tau,\prec_\tau)=
\begin{cases}
(V^0_\beta,\prec_{\beta,0})&\tau=\beta\leq\delta_0,\\
(V^0_{\delta_0}\cup V^1_\gamma,\prec_{\delta_0,\gamma})
       &\tau=\delta_0+\gamma,\quad\gamma\leq\delta_1.
\end{cases}
\]
Symmetrically, $\QQ^1\otimes\QQ^0$ first performs $\QQ^1$ and then
$\QQ^0$, with the data $(V^1_\gamma,\prec_{0,\gamma})$ and
$(V^1_{\delta_1}\cup V^0_\beta,\prec_{\beta,\delta_1})$ from the same
family.
\end{definition}

Both products are forcing equivalent to
$\QQ^0_{\delta_0}\times\QQ^1_{\delta_1}$ and have the terminal order
$\prec_{\delta_0,\delta_1}$. They are ordered coding iterations over $s$:
(O1) is immediate, (O2) follows from $Z_0\cap Z_1=\varnothing$, and (O3)
follows from Lemma~\ref{lem:amalgamation}(b),(d). For a product with a
prefix $\QQ^1\res\gamma_0$, we use the restriction of the same family to
$\gamma\leq\gamma_0$.

\begin{proposition}\label{prop:products}
For every $\alpha<\omega_1$ and every stage $s=(M,V,\prec,E)$ with
$\alpha\leq\lh(E)$:
\begin{enumerate}[label=$(\mathrm{P}_\alpha)$,leftmargin=*]
\item If $\QQ^0,\QQ^1\in\mathcal C_\alpha(s)$ are disjoint, then
$\QQ^0\otimes\QQ^1,\QQ^1\otimes\QQ^0\in\mathcal C_\alpha(s)$.
\end{enumerate}
\begin{enumerate}[label=$(\mathrm{I}_\alpha)$,leftmargin=*]
\item If $\TT\in\mathcal C_\alpha(s)$ and $G$ is $\TT$-generic over $M$,
then for every $\Sigma^1_3$ sentence $\varphi$ with real parameters from
$M$, $\Poss_\alpha(s,\varphi)\iff\Poss_\alpha(s[G],\varphi)$.
\end{enumerate}
\end{proposition}

\begin{proof}
Induct on $\alpha$, assuming $(\mathrm P_\nu)$ and $(\mathrm I_\nu)$
for $\nu<\alpha$.

$(\mathrm P_\alpha)$: we consider $\QQ^0\otimes\QQ^1$; the case of
$\QQ^1\otimes\QQ^0$ is symmetric. The first $\delta_0$ stages of
$\QQ^0\otimes\QQ^1$ are those of $\QQ^0$. Along the second factor, write
\[
s^1_\gamma:=(M_{0,\gamma},V^1_\gamma,\prec^1_\gamma,E),\qquad
t_\gamma:=(M_{\delta_0,\gamma},V^0_{\delta_0}\cup V^1_\gamma,
                              \prec_{\delta_0,\gamma},E).
\]
Clause (A2) follows from Lemma~\ref{lem:amalgamation}(c). For an active
step $\gamma$ of $\QQ^1$, (A1) follows from
\[
w^1_\gamma=\la\WO,u,v\ra
\ \Longrightarrow\ u\prec^1_\gamma v
\ \Longrightarrow\ u\prec_{\delta_0,\gamma}v,
\]
since $\prec_{\delta_0,\gamma}$ agrees with $\prec^1_\gamma$ on
$R_{0,\gamma}$; also $u\neq v$.

For (A3), suppose $a^1_\gamma=1$, $w^1_\gamma=\la\SEP,c,y,m,k,x,i\ra$ and
$\epsilon_\nu=(c,y,m,k)$ with $\nu<\alpha$. By (A3) for $\QQ^1$,
$i=\side_\nu(s^1_\gamma,x)$. The forcings $\QQ^1\res\gamma$ and $\QQ^0$
are disjoint, and they belong to $\mathcal C_\nu(s)$ by
Lemma~\ref{lem:basic}(a),(d). By $(\mathrm P_\nu)$ and
Lemma~\ref{lem:basic}(d), the tail of
$(\QQ^1\res\gamma)\otimes\QQ^0$ after stage $\gamma$ belongs to
$\mathcal C_\nu(s^1_\gamma)$. Its data at $\beta\leq\delta_0$ are
$(V^1_\gamma\cup V^0_\beta,\prec_{\beta,\gamma})$, so its terminal stage
is $t_\gamma$. Apply $(\mathrm I_\nu)$ to obtain
\[
\Poss_\nu(s^1_\gamma,x\in A_j(y))
\iff\Poss_\nu(t_\gamma,x\in A_j(y))\quad(j=m,k).
\]
Hence $i=\side_\nu(s^1_\gamma,x)=\side_\nu(t_\gamma,x)$, proving (A3).

$(\mathrm I_\alpha)$, $\Leftarrow$: choose a $\TT$-name $\dot\QQ'$ and
$t\in G$ such that
\[
t\forces_\TT
 \bigl(\dot\QQ'\in\mathcal C_\alpha(s[\dot G])
               \ \land\ \dot\QQ'\forces\varphi\bigr).
\]
Then $(\TT\res t)*\dot\QQ'\in\mathcal C_\alpha(s)$ and forces $\varphi$,
by Lemma~\ref{lem:basic}(c),(e), with padding if needed.

For $\Rightarrow$, choose $\QQ\in\mathcal C_\alpha(s)$ with
$\QQ\forces\varphi$. By Lemma~\ref{lem:relocation}, we may assume that
$\QQ$ is disjoint from $\TT$. By $(\mathrm P_\alpha)$ and
Lemma~\ref{lem:basic}(d), the tail of $\TT\otimes\QQ$ after $\TT$ belongs
to $\mathcal C_\alpha(s[G])$; its unavailable sets and orders are
$V^{\TT}_{\lh(\TT)}\cup V^{\QQ}_\gamma$ and $\prec_{\lh(\TT),\gamma}$. If
$H$ is generic over $M[G]$ for this tail, then $G\times H$ is generic
over $M$ for $\TT\times\QQ$, so $H$ is $\QQ$-generic over $M$ and
$M[H]\models\varphi$; as $\varphi$ is $\Sigma^1_3$, it also holds in
$M[H][G]=M[G][H]$. Thus this tail witnesses $\Poss_\alpha(s[G],\varphi)$.
\end{proof}

\begin{corollary}\label{cor:sides}
Let $s=(M,V,\prec,E)$ be a stage. If $\nu<\lh(E)$,
$\epsilon_\nu=(c,y,m,k)$ and $\TT\in\mathcal C_\nu(s)$, then
$\side_\nu(s,x)=\side_\nu(s[G],x)$ for every $\TT$-generic $G$ over $M$
and every $x\in\R^M$.
\end{corollary}

\begin{proof}
By $(\mathrm I_\nu)$, each of the tests $\Poss_\nu(s,x\in A_j(y))$,
$j\in\{m,k\}$, holds iff it holds over $s[G]$. Since $\side_\nu$ is
defined from these two tests, the claim follows.
\end{proof}

The proof of $(\mathrm I_\alpha)$ replaces a witness for a test by the
tail of a product. If the set of all reals of the stage is an
end-extension requirement, this tail remains
allowable when the requirement list is extended by entries of the
following two kinds. This is used in Lemma~\ref{lem:persistence}.

\begin{lemma}\label{lem:lifting}
Let $s=(M,V,\prec,E)$ be a stage and $\alpha\leq\lh(E)$ with
$\R^M\in\mathcal F_\alpha(E)$. Let $\TT,\QQ\in\mathcal C_\alpha(s)$ be
disjoint, let $G$ be $\TT$-generic over $M$ with terminal stage
$s[G]=(M[G],V^*,\prec^*,E)$, and let $\QQ^G$ be the tail of
$\TT\otimes\QQ$ after $\TT$. Suppose that
$s'=(M[G],V^*,\prec^*,E')$ is a stage with
$E'\res\alpha=E\res\alpha$ such that every entry $\epsilon'_\nu$ of $E'$
with $\alpha\leq\nu<\lh(E')$ is
\begin{enumerate}[label=(\roman*)]
\item an end-extension requirement which belongs to $\Init(\prec^*)$, or
\item a separation requirement $(c,y,m,k)$ with $c\notin M$.
\end{enumerate}
Then $\QQ^G\in\mathcal C_{\lh(E')}(s')$. Moreover, if $\varphi$ is a
$\Sigma^1_3$ sentence with real parameters from $M$ and
$\QQ\forces\varphi$, then $\QQ^G\forces\varphi$ over $M[G]$.
\end{lemma}

\begin{proof}
By the proof of $(\mathrm I_\alpha)$,
$\QQ^G\in\mathcal C_\alpha(s[G])$, and $\QQ^G$ forces $\varphi$ over
$M[G]$ whenever $\QQ$ forces $\varphi$. We use the notation introduced
before Lemma~\ref{lem:intersection}, with $\QQ^0:=\TT$, $\QQ^1:=\QQ$ and
$G^{(0)}:=G$. Thus $\QQ^G$ carries the unavailable sets
$V^*\cup V^1_\gamma$ and the orders $\prec_{\delta_0,\gamma}$,
$\gamma\leq\delta_1$, and ${\prec_{\delta_0,0}}={\prec^0_{\delta_0}}={\prec^*}$.

Clauses (O1)--(O3) and (A1) do not involve the requirement list. For
$\nu<\alpha$ we have $\epsilon'_\nu=\epsilon_\nu$, and $\side_\nu$ only
depends on the first $\nu+1$ entries of the requirement list. Hence the
instances of (A2) and (A3) for $\nu<\alpha$ are the same for $E$ and for
$E'$, and they hold since $\QQ^G\in\mathcal C_\alpha(s[G])$. It remains
to consider the entries $\epsilon'_\nu$ with $\alpha\leq\nu<\lh(E')$.

\emph{End-extension requirements.} We claim that
\begin{equation}\label{eq:end-extension}
\prec^*\unlhd\prec_{\delta_0,\gamma}\qquad(\gamma\leq\delta_1).
\end{equation}
Since $\R^M\in\mathcal F_\alpha(E)$, clause (A2) for $\TT$ and $\QQ$ gives
$R_{0,0}=\R^M\in\Init(\prec^0_\beta)\cap\Init(\prec^1_\gamma)$ for all
$\beta\leq\delta_0$ and $\gamma\leq\delta_1$. Hence, in the proof of
Lemma~\ref{lem:amalgamation}, every $0$-new and every $1$-new real has
cut $\theta$, and $\prec_{\delta_0,\gamma}$ orders the non-mixed reals by
the sum
\[
(R_{0,0},\prec)+(X^0_{\delta_0},\prec^0_{\delta_0})
               +(X^1_\gamma,\prec^1_\gamma),
\]
followed by the mixed reals. The first two summands make up
$(R_{\delta_0,0},\prec^0_{\delta_0})=(\R^{M[G]},\prec^*)$, which proves
\eqref{eq:end-extension}. If $\epsilon'_\nu$ is an end-extension requirement in
$\Init(\prec^*)$, then \eqref{eq:end-extension} yields
$\epsilon'_\nu\in\Init(\prec_{\delta_0,\gamma})$ for every
$\gamma\leq\delta_1$, which is (A2) for $\epsilon'_\nu$.

\emph{Separation requirements.} Let $\epsilon'_\nu=(c,y,m,k)$ with
$c\notin M$, and note that $c\in\R^{M[G]}=R_{\delta_0,0}$. The names for
the reals coded by $\QQ^G$ are the names of $\QQ$, lifted along the
product embedding. So if $\gamma'$ is an active step of $\QQ^G$, then
$w_{\gamma'}\in M[G^{(1)}_{\gamma'}]$, that is,
$w_{\gamma'}\in R_{0,\gamma'}$. If $w_{\gamma'}$ were of the form
$\la\SEP,c,y,m,k,x,i\ra$, then $c$, being recursive in $w_{\gamma'}$, would
belong to $R_{0,\gamma'}\cap R_{\delta_0,0}=R_{0,0}=\R^M$ by
Lemma~\ref{lem:intersection}, a contradiction. Thus (A3) holds vacuously
for $\epsilon'_\nu$.
\end{proof}

\section{The final iteration}\label{sec:iteration}

\subsection{Bookkeeping}

Fix in $L$ a function
$F\colon\omega_1\to\omega_1^3\times\omega^2\times3$ such that
\[
\forall t\in\omega_1^3\times\omega^2\times3\quad
\sup F^{-1}[\{t\}]=\omega_1,
\qquad
F(\zeta)=(\zeta_0,\xi_0,\xi_1,m,k,j)\Longrightarrow\zeta_0\leq\zeta.
\]

\subsection{The initial requirement list}\label{subsec:initial}

For the lightface separation property, the requirements for the
parameter-free pairs are fixed in $W$, before the iteration starts. Let
$((m_n,k_n):n<\omega)$ be a recursive enumeration of $\omega^2$ without
repetitions, let $\chi_n\in2^\omega$ be the characteristic function of
$\{n\}$, and let $\iota_n$ denote the sentence
\[
A_{m_n}(\mathbf 0)\cap A_{k_n}(\mathbf 0)\neq\varnothing,
\]
a $\Sigma^1_3$ sentence whose only parameter is the recursive real
$\mathbf 0$. Every pair of $\Sigma^1_3$ sets which are defined without
parameters is of the form $(A_{m_n}(\mathbf 0),A_{k_n}(\mathbf 0))$ for
some $n$, since a formula $\psi(x)$ can be regarded as a formula in the
variables $x,y$.

Fix a well-order $\prec^0\in L[G^0]=W\res\varnothing$ of
$\R^W=\R^{L[G^0]}$. In $W$ we define, by recursion on $\ell<\omega_1$,
natural numbers $n_\ell$ and requirements $\epsilon_{1+\ell}$. Put
$\epsilon_0:=\R^W$. Suppose that $n_{\ell'}$ and $\epsilon_{1+\ell'}$ have
been defined for all $\ell'<\ell$, and let
\[
E_\ell:=(\epsilon_\nu:\nu<1+\ell),\qquad
t_\ell:=(W,\varnothing,\prec^0,E_\ell).
\]
If there is $n\in\omega\setminus\{n_{\ell'}:\ell'<\ell\}$ such that
$\Poss_{1+\ell}(t_\ell,\iota_n)$ fails, let $n_\ell$ be the least such
$n$, and put
\[
\epsilon_{1+\ell}:=(\chi_{n_\ell},\mathbf 0,m_{n_\ell},k_{n_\ell}).
\]
Otherwise the recursion stops, and we put
\[
\lambda_0:=\ell,\qquad E^0:=E_{\lambda_0},\qquad
\Lambda:=\{n_\ell:\ell<\lambda_0\},\qquad
s^0:=(W,\varnothing,\prec^0,E^0).
\]
Thus a requirement for the $n$-th pair is added at a step at which no
forcing respecting the requirements added so far forces the two sets to
intersect, and the recursion continues until every remaining pair can be
forced to intersect by a forcing which respects all requirements of
$E^0$. The requirements are added one at a time, since each new
requirement shrinks the class of allowable forcings and can thereby
change the answers to the remaining tests. The first entry $\R^W$ makes
all forcings considered in these tests freeze the well-order on $\R^W$.

\begin{lemma}\label{lem:initial}
\begin{enumerate}[label=(\alph*)]
\item Every $t_\ell$ is a stage, the recursion stops at some
$\lambda_0<\omega_1$, and $s^0$ is a stage whose only end-extension requirement $\R^W$
is an initial segment of $\prec^0$.
\item $\Poss_{1+\ell}(s^0,\iota_{n_\ell})$ fails for every
$\ell<\lambda_0$.
\item $\Poss_{1+\lambda_0}(s^0,\iota_n)$ holds for every
$n\in\omega\setminus\Lambda$.
\end{enumerate}
\end{lemma}

\begin{proof}
(a) Clauses (St1)--(St3) hold for $t_\ell$, with the empty history. The
entries of $E_\ell$ are $\R^W\in L[G^0]$ and quadruples in
$(\R^W)^2\times\omega^2$ with pairwise distinct first coordinates, since
$\ell'\mapsto n_{\ell'}$ is injective. By Lemma~\ref{lem:W}(b), applied
with $V=\varnothing$, every countable sequence of elements of $L[G^0]$
which lies in $W$ belongs to $L[G^0]$. Hence
$E_\ell\in L[G^0]=N_{t_\ell}$, and (St4) holds. So the tests
$\Poss_{1+\ell}(t_\ell,\cdot)$ are defined, and by
Lemma~\ref{lem:recursion} the recursion can be carried out in $W$. As
$\ell\mapsto n_\ell$ is an injection into $\omega$, the recursion stops
at a countable ordinal.

(b),(c) The class $\mathcal C_\alpha(s)$ only depends on
$(M,V,\prec,E\res\alpha)$, and $E^0\res(1+\ell)=E_\ell$. Therefore
\[
\Poss_{1+\ell}(s^0,\varphi)\iff\Poss_{1+\ell}(t_\ell,\varphi)
\qquad(\ell\leq\lambda_0).
\]
Now (b) holds by the choice of $n_\ell$, and (c) holds because the
recursion stops at $\lambda_0$.
\end{proof}

\subsection{The segments}

In $W$ we recursively define a strictly increasing continuous sequence
$(\beta_\zeta:\zeta<\omega_1)$, finite support coding iterations
$\PP_{\beta_\zeta}$, and names for stages and surjections
\[
\begin{gathered}
s^\zeta=(M_\zeta,V^\zeta,\prec^\zeta,E^\zeta),
\qquad M_\zeta=W[G_{\beta_\zeta}],\\
e_\zeta\colon\omega_1\twoheadrightarrow\R^{M_\zeta},
\qquad e_\zeta\in M_\zeta\res V^\zeta.
\end{gathered}
\]
We first specify a forcing $\BB^\zeta$ in $M_\zeta$, and then pad it to
a length fixed in $W$, which gives the segment $\QQ^\zeta$. All choices
are made by fixed choice functions on normalized presentations and by the maximum
principle, so that they correspond to names.

\paragraph{The start.} Put
$\beta_0=0$, $M_0=W$ and $V^0=\varnothing$, and let $\prec^0$, $E^0$ and
$s^0=(M_0,V^0,\prec^0,E^0)$ be as in Subsection~\ref{subsec:initial}.

\paragraph{Orders and the list.} Given $s^\zeta$, put
\[
E^\zeta_+:=E^\zeta{}^\frown\la\R^{M_\zeta}\ra,
\qquad \alpha_\zeta:=\lh(E^\zeta_+),
\qquad s^\zeta_+:=(M_\zeta,V^\zeta,\prec^\zeta,E^\zeta_+).
\]
Since $\R^{M_\zeta}\in M_\zeta\res V^\zeta$, this is a stage; its last
entry is the end-extension requirement for the set of all reals
present.

For $\zeta=0$, the set $\R^{M_0}=\R^W$ already occurs as the entry
$\epsilon_0$ of $E^0$, and repeating it does not change the allowable
forcings. Thus $\alpha_0=1+\lambda_0+1$ and, for all $\varphi$,
\begin{equation}\label{eq:level-zero}
\mathcal C_{\alpha_0}(s^0_+)=\mathcal C_{1+\lambda_0}(s^0),\qquad
\Poss_{\alpha_0}(s^0_+,\varphi)\iff\Poss_{1+\lambda_0}(s^0,\varphi).
\end{equation}

\paragraph{The segment at a task.} The \emph{task} of segment $\zeta$
is $F(\zeta)$. In $M_\zeta$, let
\[
F(\zeta)=(\zeta_0,\xi_0,\xi_1,m,k,j),\qquad
(u,v):=(e_{\zeta_0}(\xi_0),e_{\zeta_0}(\xi_1)),\qquad
\eta^\zeta:=\min(\omega_1\setminus V^\zeta).
\]
Thus $u,v\in M_{\zeta_0}\subseteq M_\zeta$. We choose
$\BB^\zeta\in\mathcal C_{\alpha_\zeta}(s^\zeta_+)$ as follows. Whenever
$\BB^\zeta$ is a one-step forcing, it is the one-step forcing of
Lemma~\ref{lem:basic}(f) with area $\eta^\zeta$, where the order of the
new reals is chosen by a fixed rule.
\begin{enumerate}[label=(\alph*)]
\item \emph{Order codes} ($j=0$). If $u\neq v$, then $\BB^\zeta$ codes
\[
\begin{cases}
\la\WO,u,v\ra,&u\prec^\zeta v,\\
\la\WO,v,u\ra,&v\prec^\zeta u.
\end{cases}
\]
If $u=v$, then $\BB^\zeta$ is trivial.
\item \emph{Intersection tests} ($j=1$). If $E^\zeta$ contains a
separation requirement of the form $(c,u,m,k)$, then $\BB^\zeta$ is
trivial. Otherwise we ask whether
\[
\Poss_{\alpha_\zeta}
(s^\zeta_+,A_m(u)\cap A_k(u)\neq\varnothing)
\]
holds. If it does (the \emph{positive case}, denoted (b)$^+$), let
$\BB^\zeta$ be a witness. If not (the \emph{negative case}, denoted
(b)$^-$), let $\BB^\zeta$ code $\la\AUX\ra$, and let $c_\zeta$ be its
generic real.
\item \emph{Separation codes} ($j=2$). If $E^\zeta$ contains a
separation requirement $\epsilon_\nu=(c,v,m,k)$, then $\BB^\zeta$ codes
$\la\SEP,c,v,m,k,u,i\ra$, where $i:=\side_\nu(s^\zeta_+,u)$. Otherwise
$\BB^\zeta$ is trivial.
\end{enumerate}

\paragraph{The iteration as a single forcing.}
Let $\dot\BB^\zeta$ be a $\PP_{\beta_\zeta}$-name for $\BB^\zeta$.
By the ccc and Lemma~\ref{lem:padding}, choose $1\leq d_\zeta<\omega_1$
in $W$ with
$\PP_{\beta_\zeta}\forces\lh(\dot\BB^\zeta)<\check d_\zeta$, and put
\[
\dot\QQ^\zeta:=\operatorname{pad}_{d_\zeta}(\dot\BB^\zeta),\qquad
\beta_{\zeta+1}:=\beta_\zeta+d_\zeta,\qquad
\PP_{\beta_{\zeta+1}}:=\PP_{\beta_\zeta}*\dot\QQ^\zeta.
\]
In particular, $d_\zeta$ and $\beta_{\zeta+1}$ are ordinals in $W$ which do
not depend on the generic filter. Let $G^\zeta$ be the $\QQ^\zeta$-generic
filter. Padding does not change the terminal model, unavailable set and
order of $\BB^\zeta$; in case (b)$^-$, $c_\zeta$ is the generic real added
at the first step of $\QQ^\zeta$. Set
\[
\begin{gathered}
M_{\zeta+1}:=M_\zeta[G^\zeta],\qquad
(V^{\zeta+1},\prec^{\zeta+1})
:=(V^{\QQ^\zeta}_{d_\zeta},\prec^{\QQ^\zeta}_{d_\zeta}),\\
E^{\zeta+1}:=
\begin{cases}
E^\zeta_+{}^\frown\la(c_\zeta,u,m,k)\ra,&\text{in case (b)$^-$},\\
E^\zeta_+,&\text{otherwise},
\end{cases}\qquad
s^{\zeta+1}:=(M_{\zeta+1},V^{\zeta+1},\prec^{\zeta+1},E^{\zeta+1}).
\end{gathered}
\]
Choose $e_{\zeta+1}\in M_{\zeta+1}\res V^{\zeta+1}$ as above.

\paragraph{Limits.} For limit $\lambda<\omega_1$, put
\[
\begin{gathered}
\beta_\lambda:=\sup_{\zeta<\lambda}\beta_\zeta,\qquad
\PP_{\beta_\lambda}:=\varinjlim_{\zeta<\lambda}\PP_{\beta_\zeta},
\qquad M_\lambda:=W[G_{\beta_\lambda}],\\
V^\lambda:=\sup\{\eta+1:\eta\in\bigcup_{\zeta<\lambda}V^\zeta\},
\qquad E^\lambda:=\bigcup_{\zeta<\lambda}E^\zeta,\\
\prec^\lambda:=
\Big(\bigcup_{\zeta<\lambda}\prec^\zeta\Big)\mathbin{+}\lhd_\lambda.
\end{gathered}
\]
Here the ordinal $V^\lambda$ is regarded as a set of ordinals,
$\lhd_\lambda\in M_\lambda\res V^\lambda$ is a well-order of
$\R^{M_\lambda}\setminus\bigcup_{\zeta<\lambda}\R^{M_\zeta}$, and $+$
denotes the ordered sum. Choose $e_\lambda$ as before.

Finally let
$\PP:=\PP_{\omega_1}:=\varinjlim_{\zeta<\omega_1}\PP_{\beta_\zeta}$
be the finite support direct limit, and let $G$ be $\PP$-generic over $W$.
Since $d_\zeta\geq1$, we have $\beta_\zeta<\beta_{\zeta+1}$ and
$\sup_{\zeta<\omega_1}\beta_\zeta=\omega_1$, so $\omega_1$ is the
disjoint union of the intervals $[\beta_\zeta,\beta_{\zeta+1})$,
$\zeta<\omega_1$. For $\beta=\beta_\zeta+\gamma\in[\beta_\zeta,\beta_{\zeta+1})$, its
\emph{actual stage} is
\[
s_\beta:=(W[G_\beta],V^{\QQ^\zeta}_\gamma,
\prec^{\QQ^\zeta}_\gamma,E^\zeta_+).
\]
At $\beta_{\zeta+1}$, the terminal unavailable set and order of
$\QQ^\zeta$ are the initial ones of $\QQ^{\zeta+1}$, and at $\beta_\lambda$,
for limit $\lambda$, the actual stage has the unavailable set $V^\lambda$
and the order $\prec^\lambda$.

\subsection{Allowability of tails}

\begin{lemma}\label{lem:iteration-stages}
For every $\zeta<\omega_1$, the following are forced.
\begin{enumerate}[label=(\alph*)]
\item $s^\zeta$ and $s^\zeta_+$ are stages, and every end-extension
requirement of $E^\zeta_+$ is an initial segment of $\prec^\zeta$.
\item $\QQ^\zeta\in\mathcal C_{\alpha_\zeta}(s^\zeta_+)$.
\item $\zeta'<\zeta\Longrightarrow
\prec^{\zeta'}\unlhd\prec^\zeta$.
\item Distinct separation requirements in $E^\zeta$ have distinct first
coordinates, and for every $(y,m,k)$ there is at most one $c$ such that
$(c,y,m,k)$ occurs in $E^\zeta$.
\end{enumerate}
\end{lemma}

\begin{proof}
Induct on $\zeta$. For $\zeta=0$, clauses (a) and (d) follow from
Lemma~\ref{lem:initial} and the definition of $E^0$: its separation
requirements $(\chi_{n_\ell},\mathbf 0,m_{n_\ell},k_{n_\ell})$,
$\ell<\lambda_0$, have pairwise distinct first coordinates and pairwise
distinct pairs $(m_{n_\ell},k_{n_\ell})$; clause (c) is vacuous.
For the successor step, Lemma~\ref{lem:basic}(b),(f) gives
$\BB^\zeta\in\mathcal C_{\alpha_\zeta}(s^\zeta_+)$ in the trivial and
one-step cases; in case (b)$^+$ this holds by choice. In case (c), use
$\nu<\lh(E^\zeta)<\alpha_\zeta$ and (d), which makes the requirement
with first coordinate $c$ unique. Padding gives (b) for $\QQ^\zeta$
by Lemma~\ref{lem:padding}.

The terminal stage of $\QQ^\zeta$ is a stage. In case (b)$^-$,
\[
c_\zeta\in\R^{M_{\zeta+1}}\setminus\R^{M_\zeta},\qquad
c_\zeta\in M_{\zeta+1}\res V^{\zeta+1}
\]
by Lemmas~\ref{lem:coding-iteration-basic}(d)
and~\ref{lem:reconstruction}.
Thus adjoining $(c_\zeta,u,m,k)$ gives a stage $s^{\zeta+1}$ and
preserves (d), since case (b)$^-$ only occurs if $E^\zeta$ contains no
requirement of the form $(c,u,m,k)$. Since
$(E^\zeta_+)_{\alpha_\zeta-1}=\R^{M_\zeta}$, (A2) and (O3) yield
\[
\R^{M_\zeta}\in\Init(\prec^{\zeta+1}),\qquad
\prec^{\zeta+1}\res\R^{M_\zeta}=\prec^\zeta,
\qquad\prec^\zeta\unlhd\prec^{\zeta+1}.
\]
Together with the induction hypothesis, this proves (c) and shows that
all earlier end-extension requirements are initial segments of
$\prec^{\zeta+1}$. The new end-extension requirement $\R^{M_{\zeta+1}}$ is the whole field of $\prec^{\zeta+1}$,
which gives (a).

For limit $\lambda$, the monotonicity of the restricted models and the
$\sigma$-distributivity of the remaining Cohen forcing give
$(\prec^\zeta,E^\zeta:\zeta<\lambda)\in M_\lambda\res V^\lambda$.
By (c), $\bigcup_{\zeta<\lambda}\prec^\zeta$ is a well-order; hence
\[
\prec^\lambda\in M_\lambda\res V^\lambda,
\qquad
\forall\zeta<\lambda\ (\prec^\zeta\unlhd\prec^\lambda).
\]
This proves (a) and (c) at $\lambda$, and (d) holds since $E^\lambda$ is
the union of the increasing sequence $(E^\zeta:\zeta<\lambda)$.
\end{proof}

\begin{lemma}\label{lem:tails}
We regard tails of $\PP$ as ordered coding iterations, with the
unavailable sets and orders of the stages of the segments, with
$V^\lambda,\prec^\lambda$ at the limit boundaries $\beta_\lambda$, and with
the requirement list of the stage over which the tail is formed. Let
$\zeta_0<\omega_1$.
\begin{enumerate}[label=(\alph*)]
\item In $M_{\zeta_0}$,
$\beta_{\zeta_0}\leq\beta'<\omega_1$ implies
$\PP_{[\beta_{\zeta_0},\beta')}\in
\mathcal C_{\alpha_{\zeta_0}}(s^{\zeta_0}_+)$.
\item In $W[G_\beta]$, if
$\beta_{\zeta_0}\leq\beta<\beta_{\zeta_0+1}$ and
$\beta\leq\beta'<\omega_1$, then
$\PP_{[\beta,\beta')}\in
\mathcal C_{\alpha_{\zeta_0}}(s_\beta)$.
\end{enumerate}
\end{lemma}

\begin{proof}
Put $a:=\alpha_{\zeta_0}$ and $E_0:=E^{\zeta_0}_+$. For $\zeta\geq\zeta_0$
we have $E^\zeta_+\res a=E_0$; since $a$-allowability and the sides in
(A3) only depend on the requirements of index below $a$, replacing
$E^\zeta_+$ by $E_0$ changes neither. Lemma~\ref{lem:iteration-stages}(b)
and Lemma~\ref{lem:basic}(a),(d),(e) give $a$-allowability inside the
segments and across successor boundaries, where the unavailable sets and
orders agree.

At a limit boundary $\beta_\lambda$, Lemma~\ref{lem:iteration-stages}
and the construction give
\[
\begin{gathered}
\bigcup_{\zeta<\lambda}V^\zeta\subseteq V^\lambda\in L,
\qquad\prec^\lambda\in M_\lambda\res V^\lambda,\\
\forall\zeta<\lambda\ (\prec^\zeta\unlhd\prec^\lambda),
\qquad
\mathcal F_a(E_0)\subseteq\Init(\prec^\lambda).
\end{gathered}
\]
By (O3), the order at a stage inside a segment is a restriction of the
terminal order of that segment, and hence of $\prec^\lambda$. Therefore
(O1), (O3) and (A2) also hold at the limit boundaries. Each instance of
(O2), (A1) and (A3) concerns a step inside a segment and has already been
verified. Hence (a) follows by induction on $\beta'$, and (b) follows
from (a) and Lemma~\ref{lem:basic}(d), since $s_\beta$ is the stage of the
tail in (a) at $\beta$.
\end{proof}

\begin{lemma}\label{lem:aux}
Suppose case (b)$^-$ occurs at segment $\zeta$, with
$u=y$, $\rho:=\alpha_\zeta$ and
$\epsilon_\rho=(c_\zeta,y,m,k)\in\operatorname{ran}E^{\zeta+1}$.
Then:
\begin{enumerate}[label=(\alph*)]
\item No active $\gamma<\beta_\zeta$ has
$w_\gamma=\la\SEP,c_\zeta,\ldots\ra$.
\item $\Poss_\rho
(s^{\zeta+1},A_m(y)\cap A_k(y)\neq\varnothing)$ fails.
\end{enumerate}
\end{lemma}

\begin{proof}
(a) For active $\gamma<\beta_\zeta$,
$w_\gamma\in M_\zeta$, whereas
$c_\zeta\notin M_\zeta$ by
Lemma~\ref{lem:coding-iteration-basic}(d).
For (b), $\QQ^\zeta\in\mathcal C_\rho(s^\zeta_+)$, and the requirement
lists of its terminal stage and of $s^{\zeta+1}$ agree below index $\rho$.
If $\QQ'$ witnessed the test in (b), choose a $\QQ^\zeta$-name
$\dot\QQ'$ for it and a condition $q\in G^\zeta$ forcing its relevant
properties. Then
Lemma~\ref{lem:basic}(c),(e) gives
\[
(\QQ^\zeta\res q)*\dot\QQ'\in\mathcal C_\rho(s^\zeta_+),
\qquad
(\QQ^\zeta\res q)*\dot\QQ'
\forces A_m(y)\cap A_k(y)\neq\varnothing.
\]
This contradicts the negative outcome of the test at segment $\zeta$.
\end{proof}

\subsection{The parameter-free pairs}

The first lemma shows that a test which is positive over the initial
stage remains positive at every segment of the iteration.

\begin{lemma}\label{lem:persistence}
Let $\varphi$ be a $\Sigma^1_3$ sentence with real parameters from $W$.
If $\Poss_{1+\lambda_0}(s^0,\varphi)$ holds, then
$\Poss_{\alpha_\zeta}(s^\zeta_+,\varphi)$ holds in $M_\zeta$, for every
$\zeta<\omega_1$.
\end{lemma}

\begin{proof}
For $\zeta=0$ this is \eqref{eq:level-zero}. Let $\zeta\geq1$ and
$\TT:=\PP_{[0,\beta_\zeta)}$. By Lemma~\ref{lem:tails}(a),
$\TT\in\mathcal C_{\alpha_0}(s^0_+)$, and its terminal stage is
$(M_\zeta,V^\zeta,\prec^\zeta,E^0_+)$. By \eqref{eq:level-zero} and
Lemma~\ref{lem:relocation} (see the remark after the definition of
disjointness), there is $\QQ\in\mathcal C_{\alpha_0}(s^0_+)$ which forces
$\varphi$ and is disjoint from $\TT$. We apply Lemma~\ref{lem:lifting} to
$s:=s^0_+$, $\alpha:=\alpha_0$, $\TT$, $\QQ$ and $E':=E^\zeta_+$, so that
$s'=s^\zeta_+$. Its hypotheses hold:
\begin{itemize}
\item $\R^{M_0}\in\mathcal F_{\alpha_0}(E^0_+)$, and $E^\zeta_+$
end-extends $E^0_+$;
\item by Lemma~\ref{lem:iteration-stages}(a), $s^\zeta_+$ is a stage and
its end-extension requirements belong to $\Init(\prec^\zeta)$;
\item every separation requirement of $E^\zeta_+$ of index at least
$\alpha_0$ is of the form $(c_{\zeta'},u,m,k)$ for some $\zeta'<\zeta$ at
which case (b)$^-$ occurred, and
$c_{\zeta'}\notin M_{\zeta'}\supseteq W$ by
Lemma~\ref{lem:coding-iteration-basic}(d).
\end{itemize}
Hence the tail of $\TT\otimes\QQ$ after $\TT$ belongs to
$\mathcal C_{\alpha_\zeta}(s^\zeta_+)$ and forces $\varphi$ over
$M_\zeta$.
\end{proof}

\begin{lemma}\label{lem:lightface-pairs}
Let $n<\omega$.
\begin{enumerate}[label=(\alph*)]
\item If $n\notin\Lambda$, then $W[G]\models\iota_n$.
\item If $n\in\Lambda$, then $W[G]\models\neg\iota_n$.
\end{enumerate}
Thus $\Lambda$ is the set of all $n$ such that $A_{m_n}(\mathbf 0)$ and
$A_{k_n}(\mathbf 0)$ are disjoint in $W[G]$.
\end{lemma}

\begin{proof}
(a) Let $n\notin\Lambda$. By Lemma~\ref{lem:initial}(c) and
Lemma~\ref{lem:persistence},
\begin{equation}\label{eq:always-positive}
\Poss_{\alpha_\zeta}(s^\zeta_+,\iota_n)
\quad\text{holds for every }\zeta<\omega_1.
\end{equation}
Choose $\xi_0$ with $e_0(\xi_0)=\mathbf 0$ and a segment $\zeta$ with
$F(\zeta)=(0,\xi_0,0,m_n,k_n,1)$, so that $u=\mathbf 0$ at segment
$\zeta$. The list $E^\zeta$ contains no requirement of the form
$(c,\mathbf 0,m_n,k_n)$: there is none in $E^0$, as $n\notin\Lambda$ and
the enumeration of the pairs has no repetitions, and none was added at a
segment $\zeta'<\zeta$, since this only happens in case (b)$^-$, that is,
when $\Poss_{\alpha_{\zeta'}}(s^{\zeta'}_+,\iota_n)$ fails, contrary to
\eqref{eq:always-positive}. So the test of case (b) is carried out at
segment $\zeta$, and its outcome is positive by
\eqref{eq:always-positive}. Hence $\BB^\zeta$ forces $\iota_n$, so
$M_{\zeta+1}\models\iota_n$, and $W[G]\models\iota_n$ by upward
absoluteness of $\Sigma^1_3$ statements.

(b) Let $n=n_\ell$, and suppose that $W[G]\models\iota_n$. As $\iota_n$
is $\Sigma^1_3$, Lemma~\ref{lem:coding-iteration-basic}(c) and Shoenfield
absoluteness give $\beta<\omega_1$ with $W[G_\beta]\models\iota_n$, and
some $t\in G_\beta$ forces $\iota_n$. By Lemma~\ref{lem:tails}(a) and
Lemma~\ref{lem:basic}(a),(c),
\[
\PP_{[0,\beta)}\res t\in\mathcal C_{1+\ell}(s^0_+)
=\mathcal C_{1+\ell}(s^0).
\]
Thus $\Poss_{1+\ell}(s^0,\iota_n)$ holds, contradicting
Lemma~\ref{lem:initial}(b).
\end{proof}

\section{Verification}\label{sec:verification}

Let $G$ be $\PP$-generic over $W$, as in Section~\ref{sec:iteration}. By
Lemma~\ref{lem:coding-iteration-basic}(c),
$\R^{W[G]}=\bigcup_{\zeta<\omega_1}\R^{M_\zeta}$, and
Lemma~\ref{lem:coding} applies to $\PP$.

\subsection{Separation}\label{sec:separation}

Fix $y\in\R^{W[G]}$ and $m,k<\omega$ with
$W[G]\models A_m(y)\cap A_k(y)=\varnothing$.

\begin{lemma}\label{lem:requirement-exists}
There are $\zeta^\circ<\omega_1$, an ordinal $\rho$ and a real $c$ with
the following properties.
\begin{enumerate}[label=(R\arabic*)]
\item $\rho<\lh(E^{\zeta^\circ})$ and
$(E^{\zeta^\circ})_\rho=(c,y,m,k)$.
\item $\Poss_\rho
(s^{\zeta^\circ},A_m(y)\cap A_k(y)\neq\varnothing)$ fails.
\item No active $\gamma<\beta_{\zeta^\circ}$ has
$w_\gamma=\la\SEP,c,\ldots\ra$.
\end{enumerate}
If $y=\mathbf 0$ and $(m,k)=(m_n,k_n)$, one can take $\zeta^\circ=0$ and
$c=\chi_n$.
\end{lemma}

\begin{proof}
\emph{Case 1: $y=\mathbf 0$.} Let $(m,k)=(m_n,k_n)$. Then
$W[G]\models\neg\iota_n$, so $n\in\Lambda$ by
Lemma~\ref{lem:lightface-pairs}(a), say $n=n_\ell$. Put $\zeta^\circ:=0$,
$\rho:=1+\ell$ and $c:=\chi_n$. Then (R1) holds by the definition of
$E^0$, (R2) is Lemma~\ref{lem:initial}(b), and (R3) is vacuous, as
$\beta_0=0$.

\emph{Case 2: $y\neq\mathbf 0$.} Choose $\zeta_0,\xi_0$ with
$y=e_{\zeta_0}(\xi_0)$, and let $\zeta$ be least with
$F(\zeta)=(\zeta_0,\xi_0,0,m,k,1)$. If $E^\zeta$ already contains a
requirement $(c,y,m,k)$, then it does not belong to $E^0$, as
$y\neq\mathbf 0$, so it was added in case (b)$^-$ at an earlier segment.
Otherwise the test of case (b) is carried out at segment $\zeta$, and the
positive outcome would give
\[
M_{\zeta+1}\models A_m(y)\cap A_k(y)\neq\varnothing
\Longrightarrow
W[G]\models A_m(y)\cap A_k(y)\neq\varnothing
\]
by upward absoluteness of $\Sigma^1_3$ statements, a contradiction. Hence
case (b)$^-$ occurs at some segment $\zeta^*$ with $u=y$ and the given
$m,k$. Put $\zeta^\circ:=\zeta^*+1$, $\rho:=\alpha_{\zeta^*}$ and
$c:=c_{\zeta^*}$. Then (R1) holds by the definition of $E^{\zeta^*+1}$,
and (R2) is Lemma~\ref{lem:aux}(b). For (R3), no active
$\gamma<\beta_{\zeta^*}$ has $w_\gamma=\la\SEP,c,\ldots\ra$ by
Lemma~\ref{lem:aux}(a), and the only active step of segment $\zeta^*$
codes $\la\AUX\ra$.
\end{proof}

Fix $\zeta^\circ,\rho,c$ as in Lemma~\ref{lem:requirement-exists}, and put
$A_0:=A_m(y)$, $A_1:=A_k(y)$ and $b:=\beta_{\zeta^\circ}$. The requirement
lists $E^\zeta$ and $E^\zeta_+$, $\zeta\geq\zeta^\circ$, end-extend
$E^{\zeta^\circ}$; in particular
\[
\rho<\alpha_\zeta\quad\text{and}\quad(E^\zeta_+)_\rho=(c,y,m,k)
\qquad(\zeta\geq\zeta^\circ).
\]
Lemma~\ref{lem:tails}(b) and Lemma~\ref{lem:basic}(a) give
\[
b\leq\beta\leq\beta'<\omega_1
\Longrightarrow
\PP_{[\beta,\beta')}\in\mathcal C_\rho(s_\beta).
\]
Indeed, if $\beta\in[\beta_{\zeta_0},\beta_{\zeta_0+1})$, then
$\zeta_0\geq\zeta^\circ$, so the index $\alpha_{\zeta_0}$ given by
Lemma~\ref{lem:tails}(b) is greater than $\rho$. The
requirement lists of the stages $s_\beta$ ($\beta\geq b$),
$s^{\zeta^\circ}$ and $s^{\zeta^\circ}_+$ agree up to and including index
$\rho$; hence, for the same $M$, $V$ and $\prec$, they give the same
$\rho$-allowable forcings and the same values of $\side_\rho$.

\begin{lemma}\label{lem:not-both}
For $b\leq\beta<\omega_1$ and $x\in\R^{W[G_\beta]}$,
\[
\neg\bigl(\Poss_\rho(s_\beta,x\in A_0)
\wedge\Poss_\rho(s_\beta,x\in A_1)\bigr).
\]
\end{lemma}

\begin{proof}
Otherwise choose witnesses $\QQ_0,\QQ_1$; by Lemma~\ref{lem:relocation}
we may assume that they are disjoint. Proposition~\ref{prop:products}
and upward absoluteness of $\Sigma^1_3$ statements give
\[
\QQ:=\QQ_0\otimes\QQ_1\in\mathcal C_\rho(s_\beta),
\qquad\QQ\forces x\in A_0\cap A_1.
\]
By Lemma~\ref{lem:tails}(a) and Lemma~\ref{lem:basic}(a),
$\TT:=\PP_{[b,\beta)}\in
\mathcal C_\rho(s^{\zeta^\circ}_+)$, and the terminal stage of $\TT$
agrees with $s_\beta$ up to index $\rho$. Choose a $\TT$-name $\dot\QQ$
for $\QQ$ and a condition $t$ in the generic filter for $\TT$ forcing the
displayed properties. Then
Lemma~\ref{lem:basic}(c),(e) yields
\[
(\TT\res t)*\dot\QQ\in
\mathcal C_\rho(s^{\zeta^\circ}),
\qquad
(\TT\res t)*\dot\QQ\forces A_0\cap A_1\neq\varnothing,
\]
contrary to (R2).
\end{proof}

For $x\in\R^{W[G]}$, put
\[
\beta(x):=\min\{\beta\geq b:x\in W[G_\beta]\},
\qquad i(x):=\side_\rho(s_{\beta(x)},x).
\]
Corollary~\ref{cor:sides} and the agreement of the requirement lists
noted above imply
$\side_\rho(s_\beta,x)=i(x)$ for every $\beta\geq\beta(x)$.

\begin{lemma}\label{lem:sides-constant}
For every $x\in\R^{W[G]}$,
\[
\{i<2:\exists\gamma<\omega_1\,
(a_\gamma=1\ \wedge\
w_\gamma=\la\SEP,c,y,m,k,x,i\ra)\}=\{i(x)\}.
\]
\end{lemma}

\begin{proof}
Let $\gamma$ be an active step with $w_\gamma=\la\SEP,c,y,m,k,x,i\ra$. By
(R3), $\gamma\geq b$, so $\gamma$ lies in a segment $\zeta\geq\zeta^\circ$;
moreover $\gamma\geq\beta(x)$, since $x$ is recursive in
$w_\gamma\in W[G_\gamma]$. Since
$\QQ^\zeta\in\mathcal C_{\alpha_\zeta}(s^\zeta_+)$, $\rho<\alpha_\zeta$
and $(E^\zeta_+)_\rho=(c,y,m,k)$, (A3) gives
$i=\side_\rho(s_\gamma,x)=i(x)$.

For existence, choose $\zeta_2\geq\zeta^\circ$ and $\xi_0,\xi_1$ with
$(e_{\zeta_2}(\xi_0),e_{\zeta_2}(\xi_1))=(x,y)$.
At a segment $\zeta\geq\zeta_2$ with
$F(\zeta)=(\zeta_2,\xi_0,\xi_1,m,k,2)$, case (c) applies to
$\epsilon_\rho=(c,y,m,k)$, which is the only requirement of this form in
$E^\zeta$ by Lemma~\ref{lem:iteration-stages}(d), and supplies the
required active code.
\end{proof}

In $W[G]$ put
$D_i:=\{x\in\R:\Phi(\la\SEP,c,y,m,k,x,i\ra)\}$, for $i<2$.

\begin{lemma}\label{lem:separating}
In $W[G]$,
\[
D_0\mathbin{\dot\cup}D_1=\R,
\qquad A_i\subseteq D_i\quad(i<2).
\]
\end{lemma}

\begin{proof}
Lemma~\ref{lem:coding} and
Lemma~\ref{lem:sides-constant} give
\[
x\in D_i
\iff\exists\gamma<\omega_1\,
(a_\gamma=1\ \wedge\ w_\gamma=\la\SEP,c,y,m,k,x,i\ra)
\iff i=i(x).
\]
This proves the partition assertion.

Let $x\in A_i$ in $W[G]$. Write the corresponding $\Sigma^1_3$
formula as $\exists z\,\theta_i(x,y,z)$, with $\theta_i\in\Pi^1_2$,
and choose a witness $z$. Choose $\beta\geq\beta(x)$ with
$x,y,z\in W[G_\beta]$; Shoenfield absoluteness gives
\[
W[G]\models\theta_i(x,y,z)
\Longrightarrow W[G_\beta]\models\theta_i(x,y,z)
\Longrightarrow W[G_\beta]\models x\in A_i.
\]
The end-extension requirements of index below $\rho$ are initial
segments of the order
of $s_\beta$ by Lemmas~\ref{lem:iteration-stages} and~\ref{lem:tails}.
Hence the trivial forcing belongs to $\mathcal C_\rho(s_\beta)$ by
Lemma~\ref{lem:basic}(b), so $\Poss_\rho(s_\beta,x\in A_i)$ holds.
Using Lemma~\ref{lem:not-both} when $i=1$, the definition of $\side_\rho$
gives $\side_\rho(s_\beta,x)=i$. Thus $i(x)=i$ and $x\in D_i$.
\end{proof}

\begin{proposition}\label{prop:separation}
$W[G]$ satisfies $\Sig^1_3$-separation.
\end{proposition}

\begin{proof}
Let $A,B$ be disjoint $\Sig^1_3$ sets in $W[G]$. Choose $m,k,y$ with
$(A,B)=(A_m(y),A_k(y))$, and let $c,D_0,D_1$ be as above. Since $\Phi$ is
$\Sigma^1_3$ and the tuple coding is recursive, Lemma~\ref{lem:separating}
gives
\[
D_0,D_1\in\Sigma^1_3(c,y),\qquad
D_1=\R\setminus D_0,\qquad
A\subseteq D_0\subseteq\R\setminus B.
\]
Thus $D_0\in\Delta^1_3(c,y)$ separates $A$ from $B$.
\end{proof}

\begin{proposition}\label{prop:lightface}
$W[G]$ satisfies lightface $\Sigma^1_3$-separation: any two disjoint
$\Sigma^1_3$ sets of reals in $W[G]$ can be separated by a $\Delta^1_3$
set, where all three sets are defined without parameters.
\end{proposition}

\begin{proof}
Let $A,B$ be disjoint sets in $W[G]$ which are $\Sigma^1_3$ without
parameters. As noted in Subsection~\ref{subsec:initial}, there is $n$
with $(A,B)=(A_m(\mathbf 0),A_k(\mathbf 0))$, where $(m,k)=(m_n,k_n)$. By
Lemma~\ref{lem:requirement-exists}, the results of this section apply
with $y=\mathbf 0$ and $c=\chi_n$, and Lemma~\ref{lem:separating} gives
\[
D_i=\{x\in\R:\Phi(\la\SEP,\chi_n,\mathbf 0,m,k,x,i\ra)\}\quad(i<2),
\qquad D_1=\R\setminus D_0,\qquad
A\subseteq D_0\subseteq\R\setminus B.
\]
The reals $\chi_n$ and $\mathbf 0$ are recursive, and so is the tuple
coding. Hence the map $x\mapsto\la\SEP,\chi_n,\mathbf 0,m,k,x,i\ra$ is
recursive, and $D_0,D_1$ are $\Sigma^1_3$ without parameters, since $\Phi$
is. Thus $D_0$ is a lightface $\Delta^1_3$ set which separates $A$ from
$B$.
\end{proof}

\subsection{The well-order}\label{sec:wellorder}

Put $\prec^*:=\bigcup_{\zeta<\omega_1}\prec^\zeta$.

\begin{lemma}\label{lem:final-order}
In $W[G]$, $\prec^*$ is a well-order of $\R$, and for all reals
$u\neq v$,
\[
u\prec^*v\iff\Phi(\la\WO,u,v\ra).
\]
\end{lemma}

\begin{proof}
By Lemma~\ref{lem:iteration-stages}(c),
\[
\prec^\zeta\unlhd\prec^{\zeta'}\quad(\zeta<\zeta'),
\qquad\R^{W[G]}=\bigcup_{\zeta<\omega_1}\R^{M_\zeta}.
\]
Hence $\prec^*$ is a linear order of $\R^{W[G]}$. If
$(x_n:n<\omega)$ were $\prec^*$-descending with $x_0\in M_\zeta$, then all
$x_n$ would belong to $\R^{M_\zeta}$, as $\R^{M_\zeta}$ is an initial
segment of $\prec^*$, and $x_{n+1}\prec^\zeta x_n$ for all $n$,
contradicting the well-foundedness of $\prec^\zeta$.

Fix $u\neq v$. If $\Phi(\la\WO,u,v\ra)$ holds, then by Lemma~\ref{lem:coding}
$\la\WO,u,v\ra$ is coded at an active step $\beta$ in some segment $\zeta$.
If $\prec_\beta$ denotes the order of the actual stage $s_\beta$, then
(A1) and (O3) give
\[
u\prec_\beta v,\qquad
\prec_\beta=\prec^{\zeta+1}\res\R^{W[G_\beta]}
\subseteq\prec^*,
\]
so $u\prec^*v$.
Conversely, if $u\prec^*v$, choose $\zeta_0,\xi_0,\xi_1$ with
$(e_{\zeta_0}(\xi_0),e_{\zeta_0}(\xi_1))=(u,v)$ and
$\zeta\geq\zeta_0$ with
$F(\zeta)=(\zeta_0,\xi_0,\xi_1,0,0,0)$.
Then $u\prec^\zeta v$, so case (a) codes $\la\WO,u,v\ra$;
Lemma~\ref{lem:coding} gives $\Phi(\la\WO,u,v\ra)$ in $W[G]$.
\end{proof}

\begin{proposition}\label{prop:wellorder}
In $W[G]$, $\prec^*$ is a parameter-free $\Delta^1_3$-definable
well-order of $\R$.
\end{proposition}

\begin{proof}
Lemma~\ref{lem:final-order} and linearity give, for all $u,v\in\R$,
\begin{align*}
u\prec^*v&\iff u\neq v\ \wedge\ \Phi(\la\WO,u,v\ra),\\
u\prec^*v&\iff u\neq v\ \wedge\ \neg\Phi(\la\WO,v,u\ra).
\end{align*}
The right-hand sides are parameter-free $\Sigma^1_3$ and $\Pi^1_3$,
respectively, since $\Phi$ and the tuple coding are parameter-free.
\end{proof}

\begin{remark}\label{rem:good}
The set $\R^W$ is the first entry of the requirement list $E^0$, and
$\prec^0\unlhd\prec^*$ by Lemma~\ref{lem:iteration-stages}(c). Hence
\[
\R^W\in\Init(\prec^*),\qquad |\R^W|=\aleph_1,\qquad
\R^W\subsetneq\R^{W[G]},
\]
where the last assertion holds since every active step adds a new real
(Lemma~\ref{lem:coding-iteration-basic}(d)). Thus
$\operatorname{otp}(\prec^*)>\omega_1$, so $\prec^*$ is not good. By
Corollary~\ref{cor:main-intro}, $W[G]$ has no good $\Delta^1_3(a)$
well-order of the reals for any real $a$.
\end{remark}

\subsection{Proof of the main results}

\begin{proof}[Proof of Theorem~\ref{thm:main-intro}]
The model $W[G]$ is obtained from $L$ by forcing with
$\QQ^0\times\QQ^1$, followed by the ccc forcing $\PP$.
Lemmas~\ref{lem:W}(a) and~\ref{lem:coding-iteration-basic}(b) give
cardinal preservation and $W[G]\models\CH$;
Propositions~\ref{prop:separation}, \ref{prop:lightface}
and~\ref{prop:wellorder} give the remaining assertions.
\end{proof}

\begin{proof}[Proof of Corollary~\ref{cor:main-intro}]
By Theorem~\ref{thm:main-intro}, $W[G]$ satisfies $\Sig^1_3$-separation
and has a $\Delta^1_3$ well-order of the reals. Apply
Corollary~\ref{cor:sep-wo} with $k=3$.
\end{proof}

\section{Questions}\label{sec:questions}

\begin{remark}
For a pair $A_m(y),A_k(y)$ with $y\neq\mathbf 0$, the separating set
$D_0$ is $\Delta^1_3$ in the parameters $y$ and $c$, where $c$ is an
auxiliary real added by the iteration. For $y=\mathbf 0$ the tag is
recursive, and $D_0$ is defined without parameters. The argument for
$y=\mathbf 0$ uses that the requirements for the countably many
parameter-free pairs can be imposed before the iteration starts; in the
same way one can treat any countable set of reals of $W$ in place of
$\{\mathbf 0\}$. It does not apply to all reals of $W[G]$ simultaneously:
a real $y$ which is added by a segment chosen in case (b)$^+$ may come
together with separation codes which put some $x$ on both sides for a
pair with parameter $y$, for any tag which exists at that point. Thus we
do not know whether, in $W[G]$, any two disjoint $\Sigma^1_3(y)$ sets can
be separated by a $\Delta^1_3(y)$ set for every real $y$. The well-order
is lightface definable.
\end{remark}

We close with two questions suggested by the construction.

\begin{question}
Is it consistent that there is a $\Delta^1_3$-definable well-order of the
reals and that, for every real $y$, any two disjoint $\Sigma^1_3(y)$ sets
can be separated by a $\Delta^1_3(y)$ set? Does this hold in the universe
$W[G]$ constructed above?
\end{question}

\begin{question}
Is $\Sig^1_3$-separation consistent with a $\Delta^1_3$-definable
well-order of the reals and $2^{\aleph_0}>\aleph_1$?
\end{question}

\section*{Use of AI}

This paper is based on my sketch of the construction for boldface
$\Sig^1_3$-separation and the well-order, in which the well-order on all
reals present is frozen before each segment of the iteration. The proofs
were written out from this sketch with the AI system Claude Opus 5.5 (Anthropic).
In the process, several details of the sketch were worked out or changed,
among them the countable sets of unavailable coding areas carried by
stages and ordered coding iterations (Definitions~\ref{def:stage}
and~\ref{def:ordered}) and the reconstruction of coding iterations used
for the decoding (Lemmas~\ref{lem:reconstruction} and~\ref{lem:coding}).

The lightface separation property was added later. Its proof, that is,
the initial requirement list of Subsection~\ref{subsec:initial} together
with Lemmas~\ref{lem:lifting}, \ref{lem:persistence}
and~\ref{lem:lightface-pairs}, was proposed by Claude at my request.

Fact~\ref{fact:red-wo} is taken from \cite{HoPiSide}, where it is joint
work with the AI systems Claude Fable~5.2 (Anthropic) and
\mbox{GPT-6.2}~Astra (OpenAI); its proof was found by
\mbox{GPT-6.2}~Astra.

AI systems were also used for checking proofs and for revising the
exposition and the organization of the present version. I take full
responsibility for the content of the paper.


\end{document}